\documentclass[12pt,reqno]{amsart}

\usepackage[utf8]{inputenc}

\usepackage{amsmath,amsthm,amssymb,amscd}
\usepackage{emptypage}
\usepackage{enumitem}

\usepackage{hyperref} 
\hypersetup{
	linktocpage = true,
	colorlinks=true,
	linkcolor=red,
}
\usepackage{dsfont} 

\usepackage{graphicx} 
\usepackage{subcaption}

\usepackage[colorinlistoftodos]{todonotes} 

\newtheorem{theorem}{Theorem}[section]
\newtheorem{proposition}[theorem]{Proposition}
\newtheorem{lemma}[theorem]{Lemma}
\newtheorem{corollary}[theorem]{Corollary}

\theoremstyle{definition}
\newtheorem{definition}[theorem]{Definition}

\theoremstyle{remark}
\newtheorem{remark}[theorem]{Remark}
\newtheorem*{remark*}{Remark}

\newcommand{\R}{\mathbb{R}}
\DeclareMathOperator{\PV}{P.\! V.\! }
\renewcommand{\d}{\,\mathrm{d}}

\newcommand{\domain}{Q(\Omega)}
\newcommand{\energygen}{\mathcal{E}}
\newcommand{\calibgen}{\mathcal{C}}
\newcommand{\admissiblegen}{\mathcal{A}}
\newcommand{\opgen}{\mathcal{L}}
\newcommand{\energy}{\mathcal{E}_{\rm N}}
\newcommand{\calib}{\mathcal{C}_{\rm N}}
\newcommand{\lag}{G_{\rm N}}
\newcommand{\op}{\mathcal{L}_{\rm N}}
\newcommand{\neumann}{\mathcal{N}_{\rm N}}
\newcommand{\region}{\mathcal{G}}
\newcommand{\admissible}{\mathcal{A}_{\rm N}}
\newcommand{\energyloc}{\mathcal{E}_{\rm L}}
\newcommand{\calibloc}{\mathcal{C}_{\rm L}}
\newcommand{\lagloc}{G_{\rm L}}
\newcommand{\oploc}{\mathcal{L}_{\rm L}}
\newcommand{\neumannloc}{\mathcal{N}_{\rm L}}
\newcommand{\totalvar}{\mathcal{E}_{\rm NTV}}
\newcommand{\calibtotalvar}{\mathcal{C}_{\rm NTV}}
\newcommand{\calibperla}{\mathcal{C}_{\mathcal{P}_{\rm N}, \lambda}}
\newcommand{\energyfrac}{\mathcal{E}_{s, F}}
\newcommand{\gagliardo}{\mathcal{E}_s}
\newcommand{\gagliardoeps}{\mathcal{E}_{K_\varepsilon}}
\newcommand{\calibfrac}{\mathcal{C}_{s, F}}
\newcommand{\energytot}{\mathcal{E}_{\rm M}}
\newcommand{\calibtot}{\mathcal{C}_{\rm M}}
\newcommand{\optot}{\mathcal{L}_{\rm M}}
\newcommand{\perim}{\mathcal{P}_{\rm N}}
\newcommand{\calibper}{\mathcal{C}_{\mathcal{P}_{\rm N}}}
\newcommand{\lsn}{L^{1}_s(\R^n)}

\newcommand{\tmax}{T}

\numberwithin{equation}{section}

\title[Null-Lagrangians and calibrations for general nonlocal functionals]{Null-Lagrangians and calibrations for general nonlocal functionals and an application to the viscosity theory}
\author{Xavier Cabr\'e}
\address{X. Cabr\'e \textsuperscript{1,2,3}
	\newline
	\textsuperscript{1} ICREA, Pg. Lluis Companys 23, 08010 Barcelona, Spain
	\newline
	\textsuperscript{2} Universitat Polit\`ecnica de Catalunya, Departament de Matem\`{a}tiques and IMTech, 
	Diagonal 647, 08028 Barcelona, Spain
	\newline
	\textsuperscript{3} Centre de Recerca Matem\`{a}tica, Edifici C, Campus Bellaterra, 08193 Bellaterra, Spain}
\email{xavier.cabre@upc.edu}
\author{I\~{n}igo U. Erneta}
\address{I.~U. Erneta \textsuperscript{4}
	\newline
\textsuperscript{4} 
Rutgers University, Department of Mathematics, Hill Center - Busch Campus, 
110 Frelinghuysen Road, Piscataway, NJ 08854, USA}
\email{inigo.erneta@rutgers.edu}
\author{Juan-Carlos Felipe-Navarro} 
\address{J.C. Felipe-Navarro \textsuperscript{5}
\newline
\textsuperscript{5}
Universidad Complutense de Madrid, Instituto de Matem\'{a}tica Interdisciplinar and Departamento de An\'{a}lisis Matem\'{a}tico y Matem\'{a}tica Aplicada, Plaza de las Ciencias~2, 28040 Madrid, Spain}
\email{jcfelipe@ucm.es}
\thanks{
The three authors are supported by grants PID2021-123903NB-I00 and RED2018-102650-T funded by MCIN/AEI/10.13039/501100011033 and by ``ERDF A way of making Europe''.
The second author has received founding from the MINECO grant MDM-2014-0445-18-1.
The third author has been supported by the Academy of Finland and the European Research Council (ERC) under the European Union's Horizon 2020 research and innovation program (grant agreement No 818437), and is member of CADEDIF Research Group (UCM 920894).
This work is also supported by the Spanish State Research Agency, through the Severo Ochoa and Mar\'{i}a de Maeztu Program for Centers and Units of Excellence in R\&D (CEX2020-001084-M), as well as by the Catalan project 2021 SGR 00087}


\begin{document}

\begin{abstract}
In this article we build a null-Lagrangian and a calibration for general nonlocal elliptic functionals in the presence of a field of extremals.
Thus, our construction assumes the existence of a family of solutions to the Euler-Lagrange equation whose graphs produce a foliation. 
Then, as a consequence of the calibration, we show the minimality of each leaf in the foliation. 
Our model case is the energy functional for the fractional Laplacian, 
for which such a null-Lagrangian was recently discovered by us.

As a first application of our calibration, we show that monotone solutions to translation invariant nonlocal equations are minimizers.
Our second application is somehow surprising, since here ``minimality'' is assumed instead of being concluded. 
We will see that the foliation framework is broad enough to provide a proof which establishes that minimizers of nonlocal elliptic functionals are viscosity solutions.
\end{abstract}

\maketitle


\section{Introduction}

Null-Lagrangians and calibrations have played a prominent role in the Calculus of Variations,
since they provide sufficient conditions for the minimality of critical points.
Important examples are those calibrations constructed in the presence of a 
\emph{field of extremals}, i.e., a foliation by critical points.
These notions have their origin in the classical extremal field theory of Weierstrass and are a powerful tool to prove minimality of solutions to PDEs.
Especially, they have found many relevant applications in the context of minimal surfaces.

In our previous work~\cite{CabreErnetaFelipeNavarro-Calibration1} we initiated the study of calibrations for nonlocal problems. 
There, we treated the simplest nonlocal model: the energy functional for the fractional Laplacian (the Gagliardo-Sobolev seminorm). In the present paper, we extend the theory to a wide class of nonlocal functionals. 
Our main result is the construction of a calibration for the energy functional\footnote{Consistent with the notation in \cite{CabreErnetaFelipeNavarro-Calibration1}, the subindices N and L are used throughout the text to denote nonlocal and local objects, respectively.}
\[
\energy(w) := \dfrac{1}{2}\iint_{\domain} \lag(x, y, w(x), w(y)) \d x \d y,
\]
where, given a bounded domain $\Omega \subset \R^n$, we have written
\begin{equation} \label{defQOmega}
	\domain := 
	(\R^{n}\times \R^{n})\setminus(\Omega^c \times \Omega^c) =
	(\Omega \times \Omega)\cup
	(\Omega \times \Omega^c) \cup (\Omega^c \times \Omega).
\end{equation}
Here and throughout the paper $\Omega^c = \R^n \setminus \Omega$.
Since $\domain$ is invariant under the reflection 
$(x,y) \mapsto (y,x)$,
without loss of generality we may assume that the Lagrangian $\lag$ is \emph{pairwise symmetric}, i.e., it satisfies
\begin{equation}
\label{pairwise:symmetric}
\lag(x,y,a,b) = \lag(y,x,b,a),
\end{equation}
for all $(x,y) \in \domain $ and $(a,b) \in \R^2$.
We assume \eqref{pairwise:symmetric} throughout the paper.
The Lagrangian $\lag(x, y, a, b)$ 
is required to satisfy the natural ellipticity condition 
\begin{equation}
\label{Intro_lag:convex}
\partial^2_{a b} \lag(x, y, a, b)
\leq 0,
\end{equation}
on which we elaborate below (see the comments before Theorem~\ref{thm:nonlocal:calibration} and also Section~\ref{section:nonlocal}).

As in the local theory, as well as in our recent fractional Laplacian theory~\cite{CabreErnetaFelipeNavarro-Calibration1}, our calibration for $\energy$ is built in the presence of a field of extremals.
As mentioned above,
this is a one-parameter family of critical points of $\energy$ whose graphs form a foliation (see Definition~\ref{def:foliation}). 
For the construction, it suffices to have 
subsolutions and supersolutions
 on each respective side of a given extremal, a fact that is sometimes very useful.

A first application of our calibration concerns the minimality of monotone solutions to translation invariant nonlocal equations.
More precisely, we prove that
if $u$ is a solution 
(with an appropriate regularity and growth at infinity, which will depend on the Lagrangian $\lag$)
satisfying $\partial_{x_n } u > 0$ in $\R^{n}$,
then it is a minimizer
among functions~$w$ satisfying
\[ 
\lim_{\tau \to -\infty} u(x', \tau) \leq w(x', x_n) \leq \lim_{\tau \to +\infty} u(x', \tau),
\]
for $x = (x', x_n) \in \R^{n-1}\times \R$.
This result,
which is related to 
a celebrated conjecture of De Giorgi for the Allen-Cahn equation, 
was only known for those nonlocal functionals for which an existence and regularity theory of minimizers is available.
We explain
this further in Subsection~\ref{subsection:applications}.

As a second application, we show that minimizers of nonlocal elliptic functionals are viscosity solutions.
This type of result was previously known for problems where a weak comparison principle is available; see \cite{ServadeiValdinoci,KorvenpaaKuusiLindgren,BarriosMedina}.
However, we can prove it in more general scenarios by using the calibration technique; see Subsection~\ref{subsection:viscosity}.
This strategy was previously used by the first author~\cite{Cabre-Calibration} in the context of nonlocal minimal surfaces.

\subsection{Examples} \label{subsection:examples}

Our theory covers several important elliptic functionals $\energy$ given by a Lagrangian $\lag$ as above:
\begin{itemize}[leftmargin=*]
	\item The case
	\[ \lag(x, y, a, b) = \dfrac{|a-b|^p}{2p|x-y|^{n+ps}}, \]
	with $p \in [1, \infty)$ and $s \in (0,1)$,
	corresponds to the fractional $p$-Dirichlet Lagrangian, which gives rise to the fractional 
	\mbox{$p$-Laplace} equation.
	More generally, considering
	\begin{equation}
	\label{general:plaplace}
	\lag(x, y, a, b) = \frac{|a-b|^p}{2p|x-y|^{n+ps}} - \frac{1}{2|\Omega|}\mathds{1}_{\Omega \times \Omega}(x, y) 
	(F(a, x)+F(b, y)),
	\end{equation}
	we can add a reaction term in the Euler-Lagrange equation.
	For instance, if we take $p = 2$ we recover (up to a multiplicative constant) the Lagrangian associated to the fractional semilinear equation $(-\Delta)^s u = \partial_u F(u, x)$ in $\Omega$, treated in our previous work~\cite{CabreErnetaFelipeNavarro-Calibration1}. Recall the expression for the fractional Laplacian:
	\[
	(-\Delta)^s u (x) = c_{n,s} \, \PV \int_{\R^n} \frac{u(x) - u(y)}{|x- y|^{n+2s}}  \d y,
	\]
	where $c_{n,s}$ is a positive normalizing constant and $\PV$ stands for the principal value.
	
	\item The Lagrangian
	\[ \lag(x, y, a, b) = \frac{G\left(\frac{a-b}{|x-y|}\right)}{|x-y|^{n+s-1}},\]
	where $s \in (0,1)$, $G''(\tau) = (1+\tau^2)^{-(n+s+1)/2}$, and $G(0) = G'(0) = 0$,
	recovers the fractional perimeter for subgraphs; see~\cite{CozziLombardini}.
	
	\item 
	The general structure
	\[ \lag(x, y, a, b) = G(x-y, a - b), \]
	appears in the leading terms of the previous examples and gives rise to translation invariant equations.
	However, it is also of interest to treat functionals where the interactions occur only inside $\Omega$, that is, when $\lag$ is of the form  
	\[ \lag(x, y, a, b) = \mathds{1}_{\Omega\times \Omega}(x, y)\ G(x-y, a - b).\]
	These Lagrangians appear, for instance, in the macroelastic energy from Peridynamics; see~\cite{Silling-Peridynamics}.
	In this case, $G$ might be compactly supported in the $(x-y)$-variable.
	
	\item 
	The case 
	\[ 
	\lag(x, y, a, b) = -\mathds{1}_{\Omega \times \Omega}(x, y) K(x-y) \,a b + \frac{1}{2|\Omega|}
	\mathds{1}_{\Omega \times \Omega}(x, y)
	(F(a)+F(b)),
	\]
	corresponds to convolution-type operators.
	Functionals of this type appear in numerous problems, but most notably in the framework of constrained minimization (not treated in our setting); see, for instance,
	\cite{LionsMinimizationL1, BenilanBrezis, BakkerScheel, CarrilloCraigYao}
	where the first term is the interaction energy and the second one is the entropy.
	We may assume that $K$ is even, by making $\lag$ pairwise symmetric as described above.
	On the other hand, the ellipticity condition~\eqref{Intro_lag:convex} boils down to the nonnegativity of $K$.
	
\end{itemize}

\subsection{Calibrations and fields of extremals}

A fundamental problem in the Calculus of Variations is to find conditions for a function to be a minimizer of a given functional ---that we often call ``energy'', following PDE terminology.
More precisely, given a functional $\energygen\colon \admissiblegen \to \R$ 
defined on some set of admissible functions $\admissiblegen$, and given $u \in \admissiblegen$, 
one wishes to know whether $u$ minimizes $\energygen$ among competitors in~$\admissiblegen$ having the same Dirichlet condition as $u$. For nonlocal problems, given a bounded domain $\Omega$, the Dirichlet condition refers to the value of the function in all the exterior of $\Omega$, namely, in $\Omega^c = \R^n\setminus \Omega$.

One useful method to show the minimality of a given function $u\in \admissiblegen$ consists of constructing a calibration.
This is an auxiliary functional touching the energy $\energygen$ by below at $u$ and satisfying a null-Lagrangian 
equality or inequality.
More precisely:

\begin{definition}
	\label{def:calib}
	A functional $\calibgen\colon \admissiblegen \to \R$ is a \emph{null-Lagrangian} for the functional $\energygen$ and the 
admissible function
$u \in \admissiblegen$ if the following conditions hold:
	\begin{enumerate}[label= ($\mathcal{C}$\arabic*)]
		\item \label{def:calib:2} $\calibgen(u) = \energygen(u)$.
		\item \label{def:calib:3} $\calibgen(w) \leq \energygen(w)$ for all $w \in \admissiblegen$ with the same Dirichlet condition as $u$.
		\item \label{def:calib:1} $\calibgen(u) = \calibgen(w)$ for all $w \in \admissiblegen$ with the same Dirichlet condition as $u$.
	\end{enumerate}
As we will see, it is convenient to relax this last condition to the less stringent
	\begin{enumerate}[label= ($\mathcal{C}$\arabic*$'$)]
	\setcounter{enumi}{2}	
	\item \label{def:calib:1:prime} $\calibgen(u) \leq \calibgen(w)$ for all $w \in \admissiblegen$ with the same Dirichlet condition as $u$.	
	\end{enumerate}
We refer to a functional satisfying~\ref{def:calib:2},~\ref{def:calib:3}, and~\ref{def:calib:1:prime} as a \emph{calibration} for $\energygen$ and $u$.
	\end{definition}

Once a calibration is available, the minimality of $u$ among admissible functions with the same Dirichlet condition follows immediately.
For this, simply apply~\ref{def:calib:2},~\ref{def:calib:1:prime}, and~\ref{def:calib:3}, in this order.

Historically, motivated by classical problems in Mechanics and Geometry,
significant efforts 
have been 
put into rigorously understanding 
minimizers of
general functionals of the form
\begin{equation}
	\label{Intro_loc:energy}
	\energyloc(w) := \int_{\Omega} \lagloc \big(x, w(x), \nabla w(x)\big) \d x.
\end{equation}
It is well-known that 
every minimizer is a critical point of
$\energyloc$ (an \emph{extremal}) and must satisfy the associated 
Euler-Lagrange equation.
Conversely, if the Lagrangian $\lagloc(x, \lambda, q)$ is convex in the variables $(\lambda, q)$, then the functional $\energyloc$ is convex
and every critical point is a minimizer.
This convexity assumption is too restrictive for many relevant functionals, such as the Allen-Cahn energy.
For these functionals, the Dirichlet problem may admit several extremals, not all of them being minimizers.
Nevertheless,
one often has that the Lagrangian $\lagloc(x, \lambda, q)$ is convex with respect to the gradient variable $q$, which amounts to the \emph{ellipticity} of the problem.

To understand when an extremal is a minimizer, a systematic theory of calibrations has been developed for functionals $\energyloc$ of the form~\eqref{Intro_loc:energy}.
This is the \emph{extremal field theory}, which goes back to works of Weierstrass.
The key idea is to assume the existence of a family of critical points $u^t \colon \overline{\Omega} \to \R$,
with $t$ in some interval $I \subset \R$,
whose graphs do not intersect each other.
Thus, the graphs of these functions produce a foliation of a certain region $\region$ in $\R^n \times \R$, which allows to carry out a subtle convexity argument to bound the nonconvex functional by below with a calibration.

Next, we recall our definition of field for nonlocal problems, as introduced in \cite{CabreErnetaFelipeNavarro-Calibration1}:

\begin{definition}\label{def:foliation}
Given an interval~$I \subset \R$ (not necessarily bounded, nor open),
we say that a family $\{u^{t}\}_{t \in I}$ of functions $u^{t} \colon \R^n \to \R$ is a \emph{field in $\R^n$} if
\begin{itemize}
\item the function $(x, t) \mapsto u^t(x)$ is continuous in $\R^n \times I$;
\item for each $x \in \R^n$, the function $t \mapsto u^t(x)$ is $C^1$ and increasing in $I$. 
\end{itemize}
\end{definition}

Given a functional $\energygen$ acting on functions defined in $\R^n$, and given a bounded domain $\Omega\subset \R^n$, we say that $\{u^{t}\}_{t \in I}$ is a 
\emph{field of extremals}
in $\Omega$ (for $\energygen$) when it is a field in $\R^n$ and each of the functions $u^t$ is a critical point of $\energygen$ in $\Omega$.

Given a field in $\R^n$ as above, the region
\[
\region = \{ (x, \lambda) \in \R^n \times \R \colon \lambda = u^t(x) \text{ for some } t \in I \} \subset \R^n \times \R,
\]
is foliated by the graphs of the functions $u^t$, which do not intersect each other (since $u^t(x)$ is increasing in $t$). 
In particular, we can uniquely define a \emph{leaf-parameter function}
\begin{equation}
\label{Eq:DefLeafParameter}
t \colon \mathcal{G} \to I, \quad (x, \lambda) \mapsto t(x, \lambda) \quad \text{ determined by } \quad u^{t(x,\lambda)}(x) = \lambda.
\end{equation}
The function $t$ is continuous in $\region$ by the assumptions in Definition~\ref{def:foliation}.\footnote{We only sketch the argument. Assuming that $(x_n, \lambda_n) \in \mathcal{G}$ converges to $(\bar{x}, \bar{\lambda}) \in \mathcal{G}$, we prove that every accumulation point of the sequence $t_n := t(x_n, \lambda_n) \in I$ must be equal to $\bar{t} := t(\bar{x}, \bar{\lambda}) \in I$. Suppose there is a subsequence $(t_{n_{m}})_{m}$ converging to $t^{\star} \notin I$, where $t^{\star}$ could be infinite for unbounded~$I$. Taking a further subsequence, we may assume it is monotone (say) increasing, $t_{n_{m}} \uparrow t^{\star}$, the decreasing case being analogous. By the monotonicity and continuity of the field, $u^{\bar{t}}(x) < u^{\bar{t} + \varepsilon}(x) \leq u^{t_{n_{m}}}(x)$ for $\varepsilon > 0$ small, $m$ large, and $x$ close to $\bar{x}$. Letting $x = x_{n_{m}}$ and taking $m \to \infty$, by continuity and recalling that $u^{t_{n_{m}}}(x_{n_{m}}) \to u^{\bar{t}}(\bar{x})$ we conclude that $u^{\bar{t}}(\bar{x}) < u^{\bar{t}+\varepsilon}(\bar{x}) \leq u^{\bar{t}}(\bar{x})$, a contradiction. Hence no subsequence escapes $I$. If a subsequence converges to $t^{\star} \in I$, then by continuity $u^{t^{\star}}(\bar{x}) = u^{\bar{t}}(\bar{x})$ and by monotonicity $t^{\star} = \bar{t}$, hence the claim.}
We will often refer to the functions $u^t$ (or their graphs) as the ``{\it leaves}'' of the field.

Next, let us recall the fundamental result of the classical extremal field theory.
Namely, given an elliptic Lagrangian\footnote{Recall that here ellipticity means that $\lagloc(x, \lambda, q)$ is convex with respect to the gradient variable~$q$. However, for~\eqref{Intro_calib:loc} to be a calibration, a weaker condition than convexity in $q$ suffices. One needs to assume that each $u^t$ satisfies the \emph{Weierstrass sufficient condition}, namely $$ \lagloc(x,u^t(x),q) \geq \lagloc(x,u^t(x),\nabla u^t(x)) + \partial_q \lagloc(x,u^t(x),\nabla u^t(x)) \cdot \left(q-\nabla u^t(x) \right), $$ for all $x\in \Omega$, $q\in \R^n$, and $t\in I$; see~\cite{CabreErnetaFelipeNavarro-Calibration1} for more details.}~$\lagloc$ and $\{u^t\}_{t \in I}$ a smooth field of extremals  in $\Omega$, 
the functional
\begin{equation}\label{Intro_calib:loc}
	\begin{split}
		\calibloc(w) := & \int_{\Omega} \!\Big\{ \partial_q \lagloc(x, u^{t}(x), \nabla u^{t}(x)) \cdot \big(\nabla w(x) -\nabla u^{t}(x)\big)\Big\}  \Big|_{t = t(x, w(x))} \d x\\
		& \hspace{1.5cm} + \int_{\Omega} \lagloc(x, u^{t}(x), \nabla u^{t}(x))\big|_{t = t(x, w(x))} \d x,
	\end{split}
\end{equation}
is a calibration for the functional $\energyloc$ and each critical point $u^{t_0}$, $t_0\in I$.
In particular, each leaf $u^{t_0}$ 
minimizes $\energyloc$
among competitors $w$ satisfying $w = u^{t_0}$ on $\partial \Omega$
and whose graphs lie in the region $\region$.

While trying to find an analogue of \eqref{Intro_calib:loc} for the fractional Laplacian, and inspired by the work \cite{Cabre-Calibration} of the first author on the fractional perimeter,
in~\cite[Theorem~3.1]{CabreErnetaFelipeNavarro-Calibration1} we found 
a new expression for the calibration $\calibloc$.
For each $t_0 \in I$, we proved that~$\calibloc$ in~\eqref{Intro_calib:loc} can be written as
\begin{equation}\label{Intro_calib:loc:2}
	\begin{split}
		\calibloc(w) &= \int_{\Omega} \int_{u^{t_0}(x)}^{w(x)} \oploc(u^{t})(x)\Big|_{t = t(x, \lambda)} \d \lambda \d x  \\
		&\quad \quad \quad \quad  + \int_{\partial \Omega} \int_{u^{t_0}(x)}^{w(x)} \neumannloc(u^{t})(x)\big|_{t = t(x,\lambda)} \d \lambda \d \mathcal{H}^{n-1}(x) + \energyloc(u^{t_0}),
	\end{split}
\end{equation}
where $\oploc$ and $\neumannloc$ are, respectively, the Euler-Lagrange and Neumann operators associated to the functional $\energyloc$ in \eqref{Intro_loc:energy}.
As in the fractional Laplacian framework treated in~\cite{CabreErnetaFelipeNavarro-Calibration1}, our new nonlocal calibration given in Theorem~\ref{thm:nonlocal:calibration} below will be a nonlocal analogue of identity \eqref{Intro_calib:loc:2}.

While the theory of calibrations for local equations is well understood, 
there are very few papers prior to~\cite{CabreErnetaFelipeNavarro-Calibration1} dealing with nonlocal ones, which we mention next.
In~\cite{Cabre-Calibration} the first author gave a calibration for the fractional perimeter.
Independently, Pagliari~\cite{Pagliari} investigated the abstract structure of calibrations for the fractional total variation.
This last functional involves the fractional perimeter of each sublevel set of a given function.
The author succeeded in constructing a calibration to prove that the characteristic functions of halfspaces are minimizers, but other fields of extremals are not mentioned in that work. Our present work provides, as a particular case, a calibration for the fractional total variation in the presence of a general field of extremals. Moreover, we can relate our construction with the calibration for the fractional perimeter in~\cite{Cabre-Calibration} applied to each superlevel set; see Appendix~\ref{section:total:variation}. 
	
In our previous paper~\cite{CabreErnetaFelipeNavarro-Calibration1}, we constructed a calibration for the energy associated to semilinear equations involving the fractional Laplacian, that is, for energies of the form
\[
\energyfrac(w) = \dfrac{c_{n,s}}{4}\iint_{\domain} \dfrac{|w(x)-w(y)|^2}{|x-y|^{n+2s}} \d x \d y - \int_{\Omega} F(w(x)) \d x.
\]
Indeed, given $\{u^t\}_{t \in I}$ a field of extremals in $\Omega$, we showed that
\begin{align*}
	\calibfrac(w) & =  c_{n,s} \PV \iint_{\domain} \int_{u^{t_0}(x)}^{w(x)} \dfrac{u^t(x)-u^t(y)}{|x-y|^{n+2s}}\bigg|_{t = t(x, \lambda)} \d \lambda \d x \d y- \int_{\Omega} F(w(x)) \d x\\
	& \hspace{3cm} \quad\,  + \dfrac{c_{n,s}}{4}\iint_{\domain} \dfrac{|u^{t_0}(x)-u^{t_0}(y)|^2}{|x-y|^{n+2s}} \d x \d y,
\end{align*}
is a calibration for $\energyfrac$ and $u^{t_0}$, $t_0\in I$. 
We recall that the expression of $\calibfrac$ was obtained by replacing the operators~$\oploc$ and~$\neumannloc$ appearing in~\eqref{Intro_calib:loc:2} by their nonlocal counterparts.

\subsection{Main result} \label{subsection:MainResults}
Next, we present our main result, which builds a calibration for the functional 
\begin{equation}\label{Intro_gen:energy}
	\energy(w) = \dfrac{1}{2}\iint_{\domain} \lag(x, y, w(x), w(y)) \d x \d y,
\end{equation}
when the Lagrangian $\lag(x, y, a, b)$, which (without loss of generality) is assumed to be pairwise symmetric (in the sense of~\eqref{pairwise:symmetric} above),
satisfies the ellipticity condition~\eqref{Intro_lag:convex}, i.e.,
\[
	\partial^2_{a b} \lag(x, y, a, b) 
	\leq 0.
\]

We will see that~\eqref{Intro_lag:convex} guarantees
the ellipticity of the problem (or a strong comparison principle, see Appendix~\ref{section:comparison}). 
It will also ensure that the calibrating functional defined in Theorem~\ref{thm:nonlocal:calibration} below satisfies property~\ref{def:calib:3} in Definition~\ref{def:calib}, thus mirroring the effect of ellipticity in the local case.

As in the classical theory,
in this general nonlocal framework every extremal is a minimizer whenever the functional $\energy$ is convex.
A sufficient condition to guarantee the convexity of $\energy$ is 
the Lagrangian
$(a, b) \mapsto \lag(x,y, a, b)$ 
being convex.
However, contrary to the local case, this hypothesis does not guarantee the ellipticity assumption~\eqref{Intro_lag:convex}.\footnote{\label{Footnote:Convex} Ellipticity reads as $\partial^2_{ab} \lag \leq 0$, while convexity amounts to the conditions $\partial^2_{a a} \lag \geq 0$ and $\partial^2_{a a} \lag \partial^2_{b b} \lag \geq (\partial^2_{a b} \lag )^2$. For the simple quadratic example $\lag = \mathds{1}_{\Omega \times \Omega}(x,y) (a+b)^2$, the reader can easily check that the functional is convex but not elliptic. Another example corresponding to a more interesting equation is given by the Lagrangian $\lag = \frac{1}{2}K(x-y) (a-b)^2 + \frac{1}{2\varepsilon |\Omega|}\mathds{1}_{\Omega \times \Omega}(x,y) (a+b)^2$, with $\varepsilon > 0$ small, where $K$ is the singular kernel of the fractional Laplacian. Here, the Euler-Lagrange equation is $(-\Delta)^s u(x) + \frac{1}{\varepsilon} \left(u(x) + \frac{1}{|\Omega|}\int_{\Omega} u(y) d y\right) = 0$ for $x$ in $\Omega$.}
This seems to be due to the great generality of the nonlocal Lagrangian in~\eqref{Intro_gen:energy}.
In this direction, in most examples that we have in mind, the Lagrangian has a leading term of the form $\lag(x, y, a - b)$. In this case, ellipticity is equivalent to convexity in the $(a- b)$-variable (and both reduce to $\partial^2_{(a-b),(a-b)}\lag > 0$).

Now, by adding ``lower order terms'' to $\lag(x,y,a-b)$ (such as reaction terms) we may produce nonlocal elliptic functionals which are not convex.
For instance, consider the linear equation 
$(-\Delta)^s u = \lambda u$, with $\lambda \in \R$.
This equation admits an energy functional
\[
\frac{c_{n,s}}{4}\iint_{\domain} \frac{|u(x)-u(y)|^2}{|x-y|^{n+2s}} \d x \d y - \frac{\lambda}{2} \int_{\Omega} u(x)^2 \d x,
\]
which corresponds to a Lagrangian $\lag$ as in our example~\eqref{general:plaplace}.
It is elliptic in the sense of \eqref{Intro_lag:convex} for all $\lambda \in \R$, but not convex when $\lambda$ is large enough.
Notice that the equation satisfies the strong comparison principle, 
while the availability of the weak comparison principle depends on $\lambda$.\footnote{We say that an operator $\opgen$ satisfies the \emph{strong comparison principle} if, given two functions $u$ and~$v$ satisfying $\opgen u \leq \opgen v$ in $\Omega$, $u \leq v$ in $\R^n$, and touching somewhere in $\Omega$, then $u \equiv v$ in $\R^n$. By contrast, $\opgen$ satisfies the \emph{weak comparison principle} if, given two functions $u$ and~$v$ satisfying $\opgen u \leq \opgen v$ in $\Omega$ and $u \leq v$ in $\Omega^c$, then $u \leq v$ in $\Omega$.}

For the functional $\energy$ and its associated Euler-Lagrange operator 
\begin{equation}
\label{def:opnonlocal}
\op(w)(x) := \int_{\R^n} \partial_{a} \lag(x, y, w(x), w(y)) \d y,
\end{equation}
to be well defined for $x\in \Omega$,
one needs to make growth and regularity assumptions on the Lagrangian $\lag$. 
These determine the class of admissible functions; see \cite{Bucur} for some examples of natural assumptions.
In this respect, our main result (Theorem~\ref{thm:nonlocal:calibration} below), which gives a calibration for general nonlocal Lagrangians satisfying the ellipticity hypothesis~\eqref{Intro_lag:convex}, 
is a formal result since it does not specify the precise class of admissible functions. 
In other words, the great generality of the functionals does not allow for specifying the growth and regularity assumptions on $\lag$ and on the admissible functions $w$. 
Thus, the theorem cannot take into account integrability issues.\footnote{This is in contrast with Theorem~1.3 in~\cite{CabreErnetaFelipeNavarro-Calibration1},
where we gave a fully rigorous result for the fractional Laplacian.}
However, we could give completely rigorous results for some specific families of Lagrangians, adapting the admissible class of functions to the concrete problem.
Indeed, within the proof of the next theorem, there are only a few points that must be justified, namely, the interchange of certain integrals and the convergence of some expressions.
Hence, in the following statement we use the term ``sufficiently regular for $\lag$'' in the sense that those functions make all integrals to be well defined.

Recall \eqref{defQOmega} for the meaning of $Q(\Omega)$, Definition~\ref{def:foliation} for the notion of field, and~\eqref{Eq:DefLeafParameter} for the leaf-parameter function $t$.
The properties~\ref{def:calib:2},~\ref{def:calib:3}, \ref{def:calib:1}, and~\ref{def:calib:1:prime} have been introduced in Definition~\ref{def:calib}.

\begin{theorem}
\label{thm:nonlocal:calibration}
Let $I \subset \R$ be an interval 
and let $\Omega \subset \R^n$ be a bounded domain.
Given a  function $\lag = \lag(x, y, a, b)$,
with $\lag(x,y,a,b) = \lag(y,x,b,a)$,
satisfying the ellipticity condition~\eqref{Intro_lag:convex},
let  $\{u^t\}_{t \in I}$ be 
a field in $\R^n$ (in the sense of Definition \ref{def:foliation})
which is sufficiently regular for $\lag$.

Given $t_0 \in I$, let $\energy$ be defined by \eqref{Intro_gen:energy} and $\calib$ be  the functional
\[
\calib(w) := \iint_{\domain}\! \int_{u^{t_0}(x)}^{w(x)} 
\partial_{a} \lag (x, y, u^{t}(x), u^{t}(y)) \big|_{t = t(x, \lambda)} \d \lambda \d x \d y  + \energy(u^{t_0})
\]
defined in a set $\admissible$ of sufficiently regular admissible functions $w \colon \R^n \to \R$ (for $\lag$) which satisfy ${\rm graph }\,  w \subset \region$, where
\[
\region = \big\{(x, \lambda) \in \R^n \times \R \colon \lambda = u^t(x) \quad \text{ for some } t \in I \big\}.
\]

Taking $\calibgen = \calib$ and $\energygen = \energy$ in Definition~\ref{def:calib},
we have the following:
	\begin{enumerate}[label={\rm(\alph*)}]
		\item $\calib$ satisfies~{\rm\ref{def:calib:2}} and~{\rm\ref{def:calib:3}} with $u = u^{t_0}$.	
		\item Assume in addition that the family $\{u^t\}_{t \in I}$ satisfies
		\[ 
		\begin{split}
			\op (u^{t})\geq 0 \quad \text{ in } \Omega 
			\quad
			\text{ for } t \geq t_0,
			\\
			\op (u^{t}) \leq 0 \quad \text{ in } \Omega \quad \text{ for } t \leq t_0,
		\end{split}
		\]	
		where $\op$ is the Euler-Lagrange operator associated to $\energy$ given by \eqref{def:opnonlocal}. 
		Then, $\calib$ satisfies {\rm\ref{def:calib:1:prime}} with $u = u^{t_0}$.
		In particular,  $\calib$ is a calibration for $\energy$ and $u^{t_0}$, and hence $u^{t_0}$ minimizes $\energy$ among functions $w$ in $\admissible$ such that $w \equiv u^{t_0}$ in $\Omega^c$.
		
		\item Assume in addition that $\{u^t\}_{t \in I}$ is a field of extremals in $\Omega$, that is,
		a field in~$\R^n$ satisfying
		\[ 
		\begin{split}
			\op (u^{t})= 0 \quad \text{ in } \Omega 
			\quad
			\text{ for all } t \in I.
		\end{split}
		\]
		Then, the functional $\calib$ satisfies {\rm\ref{def:calib:1}} with $u = u^{t_0}$.
		Therefore, $\calib$ is a null-Lagrangian for $\energy$ and $u^{t_0}$.	
		As a consequence, for every $t \in I$, the extremal~$u^t$ minimizes $\energy$ among functions $w$ in $\admissible$ such that $w \equiv u^{t}$ in $\Omega^c$.
	\end{enumerate}
\end{theorem}

As mentioned above, the class of functionals $\energy$ of the form~\eqref{Intro_gen:energy} 
satisfying the ellipticity condition~\eqref{Intro_lag:convex}
includes the Gagliardo-Sobolev seminorm (for which we constructed a calibration in~\cite{CabreErnetaFelipeNavarro-Calibration1}), the fractional total variation (see Appendix~\ref{section:total:variation}), and the examples in Subsection~\ref{subsection:examples}.
Our calibration in Theorem~\ref{thm:nonlocal:calibration} is a generalization of the one in~\cite{CabreErnetaFelipeNavarro-Calibration1}.
To guess the expression of~$\calib$ above, we extrapolated our new identity~\eqref{Intro_calib:loc:2} in the local theory.
The key point is that each of the terms in \eqref{Intro_calib:loc:2} has a clear nonlocal counterpart; see~\eqref{gen:calib:sym} below.

An interesting feature of the calibrations considered in this paper is their stability under the addition of functionals.
Due to their special structure, calibrations given in terms of fields can be added together to obtain new ones. 
In particular, the local theory can be combined with the nonlocal one developed in this work to produce calibrations for energies involving both local and nonlocal interaction terms.
We explain this further in~Section~\ref{section:compound}.

\subsection{An application to monotone solutions}
\label{subsection:applications}

Our interest in fields of extremals came from the study of \emph{monotone solutions} to the fractional Allen-Cahn equation
\begin{equation}
\label{eq:allen:cahn}
(-\Delta)^s u = u - u^3 \quad \text{ in } \R^{n};
\end{equation}
see~\cite{CabreSireII,CabreSolaMorales}, as well as~\cite{CozziPassalacqua} for more general integro-differential operators.
When the operator is the classical Laplacian, these solutions are related to a celebrated conjecture of De Giorgi; see~\cite{CabrePoggesi}.

In \cite[Corollary~1.4]{CabreErnetaFelipeNavarro-Calibration1}, we proved that monotone solutions of \eqref{eq:allen:cahn} are minimizers among competitors taking values in a precise region of space (the region specified in the next corollary).
Thanks to Theorem~\ref{thm:nonlocal:calibration} of the current paper, the same proof allows to establish the minimality of monotone solutions to more general nonlocal translation invariant equations. 
More precisely, given a Lagrangian 
of the form $\lag = \lag(x - y, a, b)$, with associated energy functional $\energy$ defined by~\eqref{Intro_gen:energy}, the Euler-Lagrange operator~$\op$ given by~\eqref{def:opnonlocal} 
is translation invariant, that is,
for all $x$ and $z$ in $\R^{n}$ the identity
\[ 
\op(w)(x + z) =  \op(w(\cdot + z))(x)
\]
holds.
We then have the following:

\begin{corollary}
\label{cor:applications}
Let $\lag = \lag(x - y, a, b)$ be a  function satisfying $\lag(x-y,a,b) = \lag(y-x,b,a)$ and the ellipticity condition~\eqref{Intro_lag:convex}.
Let $u$ be a sufficiently regular solution for $\lag$ 
(see the comments before Theorem~\ref{thm:nonlocal:calibration}) of $\op(u) = 0$ in $\R^{n}$, with $\op$ as in~\eqref{def:opnonlocal}.
Assume that $u$ is increasing in the $x_n$-variable, i.e.,
\begin{equation}\label{monotone}
\partial_{x_n} u > 0 \quad \text{ in } \R^n.
\end{equation}

Then, for each bounded domain $\Omega \subset \R^{n}$, $u$ is a minimizer of $\energy$ among sufficiently regular admissible functions $w$ satisfying 
\[
\lim_{\tau \to - \infty} u(x', \tau) \leq w(x', x_n) \leq \lim_{\tau \to +\infty} u(x', \tau) \quad \text{ for all } x = (x',x_n) \in \Omega,
\]
and such that $w \equiv u$ in $\Omega^c$.
\end{corollary}

This minimality result was already known for reaction equations involving the fractional Laplacian. 
For such equations, it can be proven with an alternative argument (described in the Introduction of~\cite{CabreErnetaFelipeNavarro-Calibration1}) which does not use any calibration. 
This alternative proof requires an existence and regularity theorem for minimizers,
as explained in Appendix~\ref{section:comparison}.
However, such a result is not available for many general Lagrangians of the form $\lag(x-y, a, b)$.
For these functionals, 
Corollary~\ref{cor:applications} 
allows to establish the minimality of monotone solutions for the first time. 

Notice that, given a monotone solution, the translation invariance of the equation is all what is needed in order to produce a field of extremals (by sliding the solution in the $x_n$-variable).
Therefore, Corollary~\ref{cor:applications} also holds 
for translation invariant equations involving both local and nonlocal terms; 
see Section~\ref{section:compound}.

\subsection{An application to the viscosity theory}
\label{subsection:viscosity}

Here we are interested in conditions to ensure that minimizers, or more generally weak solutions, are viscosity solutions.
Weak and viscosity solutions are different notions of solutions, both for differential and for nonlocal equations.
Within the Calculus of Variations, it is natural to work with \emph{weak solutions} belonging to the energy space.
Instead, when dealing with fully nonlinear equations, it is more suitable to work with \emph{viscosity solutions}.
Here, the equation is transferred to act on smooth functions touching the extremal from one side.

In the local framework, it has been shown that minimizers of many relevant functionals are viscosity solutions.
For the $p$-Laplace equation $- \Delta_p u = 0$ (here every weak solution is a minimizer), Juutinen, Lindqvist, and Manfredi~\cite{JuutinenLindqvistManfredi}
obtained the result by using a weak comparison principle.
This principle allows to compare the minimizer with a function touching it by below and which is later slid upwards, 
forcing the equation to have the correct sign.
For local functionals of the form~\eqref{Intro_loc:energy}, assuming convexity (a stronger condition than ellipticity), Barron and Jensen~\cite{BarronJensen} found a simpler variational argument. We comment on their strategy at the end of the present subsection as well as in Remark~\ref{remark:barron}.
Showing that non-minimizing weak solutions are viscosity solutions has also been treated in the literature.
For instance, this has been done by Medina and Ochoa~\cite{MedinaOchoa} for semilinear equations driven by the $p$-Laplacian. Their proof again uses a comparison principle.

Concerning nonlocal problems, the first results in this direction appeared in the geometric setting.
Caffarelli, Roquejoffre, and Savin~\cite{CaffarelliRoquejoffreSavin} showed that minimizers to the nonlocal perimeter are viscosity solutions of the homogeneous nonlocal mean curvature equation.
Their proof is quite involved and uses a comparison principle.
Later, Cabr\'{e}~\cite{Cabre-Calibration} was able to show the same result via a simpler calibration argument (here we will give the analogue of this result in the functional setting\footnote{We use ``functional setting'' in contrast to the geometric setting of energies related to the fractional perimeter.}).
The case of nonlocal minimal graphs has also been treated by Cozzi and Lombardini~\cite{CozziLombardini}. In the functional setting, as far as we know,
the first nonlocal result appeared in the work of Servadei and Valdinoci~\cite{ServadeiValdinoci} for linear equations involving the fractional Laplacian.
There, the authors employ a regularization by convolution that is not available for other operators.
For equations driven by the fractional $p$-Laplacian, 
we mention the paper by Korvenp\"{a}\"{a}, Kuusi, and Lindgren~\cite{KorvenpaaKuusiLindgren} where they treat the homogeneous problem, and the work by Barrios and Medina~\cite{BarriosMedina} for the semilinear one. In both cases, a comparison principle is needed.

Next, we state the main result of this subsection. 
We will show that every minimizer of our elliptic nonlocal functionals is a viscosity solution.
In contrast with most of the previous works, the novelty of our result is that we do not need a weak comparison principle, allowing us to treat a bigger class of Lagrangians.
This is achieved by a calibration argument.
In a way, the information provided by the weak comparison principle 
follows from the properties \ref{def:calib:2}-\ref{def:calib:1} satisfied by the calibration.
Recall that, as explained at the beginning of Subsection~\ref{subsection:MainResults}, the weak comparison principle does not follow from ellipticity.
However, the ellipticity of the Lagrangian (condition \eqref{Intro_lag:convex} above) suffices for the calibration argument in our proof.

Our theorem applies to general nonlocal elliptic functionals of the form~\eqref{Intro_gen:energy}. 
Since we do not make any growth and regularity assumptions on the Lagrangian~$\lag$, 
as in the main theorem above, 
our result is only formal. 
Nevertheless, again, we could give a fully rigorous statement for specific families of Lagrangians. 
In fact, this is what we do in the first part of Section~\ref{section:viscosity} for semilinear equations involving the fractional Laplacian.

\begin{theorem}
\label{thm:viscosity}
Let $\lag = \lag(x,y,a,b)$ be 
a function with $\lag(x,y,a,b) = \lag(y,x,b,a)$
satisfying 
the ellipticity condition~\eqref{Intro_lag:convex},
and let $\Omega \subset \R^n$ be a bounded domain.
Let $u$ be a sufficiently regular minimizer of the functional~$\energy$ given by~\eqref{Intro_gen:energy}.

Then, the function $u$ is a viscosity solution of the associated Euler-Lagrange equation $\op(u) = 0$ in $\Omega$.
\end{theorem}

Later in Section~\ref{section:viscosity}, we will give a more precise statement of this result, showing that minimizers by above (or by below) are viscosity supersolutions (respectively, subsolutions). 
Furthermore,
while our theorem only applies to minimizers,
we will explain how it can be used
to prove that certain non-minimizing weak solutions are viscosity solutions. 
Here, the idea will be to ``freeze'' the lower order terms; see Remark~\ref{remark:trick}.

The proof of Theorem~\ref{thm:viscosity} is based on the following energy comparison result for ordered functions embedded in a \emph{weak field} (that is, a ``degenerate field'' where the leaves are still ordered, but may touch each other; see Figure~\ref{Fig:Dibujo_viscosity} and Definition~\ref{def:visc:field}). Thus, here we will need to extend the above theory of nonlocal calibrations to the more general setting of weak fields.

\begin{theorem}
\label{thm:energy:comparison}
Let $\lag = \lag(x,y,a,b)$ be a 
function with $\lag(x,y,a,b) = \lag(y,x,b,a)$
satisfying the ellipticity condition~\eqref{Intro_lag:convex}.
Given a bounded domain $\Omega \subset \R^n$, let $u$ 
belong to $C(\overline{\Omega})$.

Assume that there exists a weak field $\{u^{t}\}_{t \in [0,\tmax]}$ for $u$ 
in $\Omega$ (in the sense of Definition~\ref{def:visc:field}) which is sufficiently regular for $\lag$ 
(see the comments at the beginning of Subsection~\ref{subsection:mixed}).
	
Then, 
if $\energy(u) < \infty$, we have
\[
\energy(u^T)\leq \energy(u) + \int_{\Omega} \int_{u(x)}^{u^T(x)} \op\,(u^{t}(x)) \big|_{t = t(x, \lambda)} \d \lambda \d x.
\]
\end{theorem}

In Section~\ref{section:viscosity}, we will prove 
analogous results to Theorems~\ref{thm:viscosity} and~\ref{thm:energy:comparison}
in the fractional semilinear setting, giving in this case fully rigorous statements under precise regularity assumptions (see Theorems~\ref{thm:viscosity:frac} and~\ref{thm:energy:comparison:frac} respectively).
Furthermore, the same proof will allow us to prove 
Theorems~\ref{thm:viscosity} and~\ref{thm:energy:comparison}
in the more general setting of ``mixed'' functionals involving both local and nonlocal terms (see Theorems~\ref{thm:gen:visc:min} and~\ref{thm:energy:comparison_combined} below).

The energy inequality in Theorem~\ref{thm:energy:comparison} is new even in the local case.
We will prove it by means of the calibration arguments developed in our previous work~\cite{CabreErnetaFelipeNavarro-Calibration1},
similarly to how we established identity (3.9) in that paper.

Once the energy inequality is available, we can prove Theorem~\ref{thm:viscosity}.
Indeed, assume that $u$ is a minimizer and that a smooth function $\varphi$ touches $u$ by below at some contact point. 
Now, we slide $\varphi$ upwards and take the maxima with $u$ to obtain a weak field.
The energy comparison with $u^{\tmax} = \max\{u,\varphi + T\}$ will show that $\varphi$ must be a supersolution at the contact point, since otherwise $u$ would not be a minimizer.
Applying the same procedure to smooth functions touching $u$ by above, we will conclude that 
$u$ is a viscosity solution.

\begin{figure}
	\centering
	\definecolor{caac8ff}{RGB}{200,0,70}

\begin{tikzpicture}[y=0.80pt, x=0.80pt, yscale=-4.000000, xscale=4.000000, inner sep=0pt, outer sep=0pt]
	\path[scale=0.265,fill=caac8ff,line cap=round,miter limit=4.00,fill
	opacity=0.186,line width=0.054pt] (488.8782,206.2692) .. controls
	(483.9369,206.2378) and (477.4633,208.9727) .. (474.0150,212.4207) .. controls
	(464.6469,221.2231) and (456.8338,230.2743) .. (444.9096,235.5660) .. controls
	(438.5223,238.4415) and (430.8243,239.7592) .. (424.4297,236.5213) .. controls
	(415.5771,232.2875) and (408.4700,223.9944) .. (400.3429,218.6734) .. controls
	(397.2127,217.3116) and (393.6112,220.7186) .. (391.8904,222.9753) .. controls
	(386.4238,230.6900) and (383.7765,239.8891) .. (380.7426,248.7527) .. controls
	(372.6370,274.2468) and (367.2243,300.9051) .. (362.1241,327.1361) .. controls
	(359.7296,339.8338) and (357.0404,353.6358) .. (355.0671,366.3889) .. controls
	(359.4129,365.6930) and (366.0153,363.4775) .. (370.2115,362.0694) .. controls
	(380.1319,358.6109) and (388.2399,353.8666) .. (397.1302,348.3229) .. controls
	(405.7811,343.6945) and (415.4394,341.5759) .. (424.9569,339.5126) .. controls
	(441.9340,336.2244) and (458.0452,334.5846) .. (476.2285,335.2913) .. controls
	(482.6954,335.9016) and (492.1579,338.2281) .. (495.6672,342.0974) .. controls
	(496.2249,341.1513) and (496.2851,339.5206) .. (496.7413,339.0095) .. controls
	(501.8788,327.1982) and (512.1696,318.1703) .. (522.6425,311.1358) .. controls
	(535.1321,302.8343) and (549.0282,296.8598) .. (562.4884,290.3822) .. controls
	(561.0117,283.1263) and (558.0319,275.4768) .. (554.7628,268.8505) .. controls
	(546.4734,252.6228) and (535.8817,237.5713) .. (522.8185,224.8472) .. controls
	(514.8928,217.4354) and (506.0071,209.8376) .. (495.2859,207.1548) .. controls
	(493.2800,206.6731) and (490.9337,206.2166) .. (488.8782,206.2693) -- cycle;
	
	\path[draw=red,line join=miter,line cap=butt,line width=0.8pt]
	(100.1411,94.9037) .. controls (100.1411,94.9037) and (104.6883,70.6944) ..
	(112.5461,71.7898) .. controls (120.4038,72.8852) and (125.4509,55.7896) ..
	(130.6956,64.8736) .. controls (135.2162,72.7035) and (138.3711,74.7726) ..
	(144.9470,78.5692);
	
	\path[draw=red,line join=miter,line cap=butt,line width=0.8pt]
	(107.1020,91.2681) .. controls (107.1020,91.2681) and (108.7626,81.5237) ..
	(114.4569,75.9993) .. controls (120.9844,69.6665) and (135.5056,71.5455) ..
	(140.8269,80.7622);
	
	\path[draw=red,line join=miter,line cap=butt,line width=0.8pt]
	(114.1827,89.4517) .. controls (114.1827,89.4517) and (116.2868,82.8638) ..
	(121.0075,80.6511) .. controls (126.7943,77.9387) and (125.9321,82.8166) ..
	(128.6964,82.8811) .. controls (131.7583,82.9524) and (134.7481,80.0324) ..
	(136.6051,83.4297);

	
	\path[draw=black,line join=miter,line cap=butt,miter limit=4.00,line
	width=0.056pt] (71.4942,101.2791) -- (193.6364,101.2791);
	
	\path[draw=black,line join=miter,line cap=butt,miter limit=4.00,line
	width=0.056pt] (93.9448,96.9404) -- (93.9448,120) -- (151.9423,120)
	-- (151.9423,75.3170);
	
	\path[draw=red,line join=miter,line cap=butt,line width=0.8pt]
	(93.9407,97.0112) .. controls (93.9407,97.0112) and (100.5835,54.0001) ..
	(106.2679,58.1193) .. controls (111.9523,62.2386) and (114.2348,67.6007) ..
	(125.4735,56.2620) .. controls (132.7794,48.8910) and (147.3463,68.1121) ..
	(148.8240,76.8561);
	
	\path[draw=black,line join=miter,line cap=butt,line width=1.8pt]
	(159.9981,63.4349) .. controls (156.5569,66.5078) and (157.3679,72.8556) ..
	(151.3462,75.5871) .. controls (142.3414,79.6719) and (133.1844,83.8538) ..
	(131.1248,90.4665) .. controls (128.0258,86.7793) and (109.2976,89.3569) ..
	(104.6672,92.3901) .. controls (99.9274,95.4950) and (95.3903,96.8029) ..
	(91.9212,97.3186) .. controls (91.9212,90.3047) and (89.9271,90.4555) ..
	(89.4973,92.9579) .. controls (89.3758,93.6648) and (89.0732,95.2381) ..
	(88.5244,94.4959) .. controls (88.2395,94.1106) and (88.4424,96.8555) ..
	(87.9284,95.4205) .. controls (87.7311,94.8696) and (87.4184,98.1083) ..
	(86.9705,96.5945) .. controls (86.8662,96.2422) and (86.4776,97.8210) ..
	(86.1931,96.8947) .. controls (85.8561,95.7976) and (85.6340,95.8674) ..
	(85.2338,96.6210) .. controls (84.5358,97.9354) and (85.1018,94.9014) ..
	(84.5715,95.5451) .. controls (84.4070,95.7447) and (83.8303,97.5830) ..
	(83.8303,96.1581) .. controls (83.8303,94.5506) and (83.2753,95.7046) ..
	(83.1034,95.8924) .. controls (82.2536,96.8209) and (82.3373,95.8421) ..
	(82.2412,95.5566) .. controls (82.0225,94.9062) and (81.7361,94.6762) ..
	(81.2581,95.1476) .. controls (80.9634,95.4383) and (80.0282,96.8440) ..
	(80.1942,94.6752) .. controls (80.3092,93.1727) and (79.5025,93.6054) ..
	(79.1235,94.2277) .. controls (78.3129,95.5591) and (78.1818,94.6400) ..
	(78.0828,93.6394);
	
	\path[draw=black,line join=miter,line cap=butt,miter limit=4.00,line
	width=1.8pt] (178.8626,45.8508) .. controls (177.1041,47.0027) and
	(177.4448,62.4066) .. (173.7150,48.6158) .. controls (169.4022,53.3465) and
	(169.8278,57.4468) .. (168.9150,68.3650) .. controls (164.8353,63.0101) and
	(164.7708,48.6982) .. (163.4849,46.8036) .. controls (160.6125,48.1372) and
	(159.6949,50.9420) .. (159.6949,50.9420);
	
	\path[draw=black,line join=miter,line cap=butt,miter limit=4.00,line
	width=0.056pt] (107.1083,91.1806) -- (107.1083,109) --
	(141.0254,109) -- (140.8269,80.7622);

	\path[fill=black,line width=0.056pt] (113.9271,60.9272) node[above right]
	(text7396) {\color{red} $u^{\tmax}$};
	
	\path[fill=black,line width=0.056pt] (117.2770,95.4030) node[above right]
	(text7400) {$u=u^{0}$};
	
	\path[fill=black,line width=0.056pt] (127.0770,80.4030) node[above right]
	(text7400) {\color{red} $u^{t_1}$};
	
	\path[fill=black,line width=0.056pt] (124.0770,71.4030) node[above right]
	(text7400) {\color{red} $u^{t_2}$};
	
	\path[fill=black,line width=0.056pt] (106.0770,71.4030) node[above right]
	(text7400) {\color{red} $u^{t_3}$};
	
	\path[fill=black,line width=0.056pt] (122.0770,125.4030) node[above right]
	(text7400) {$\Omega$};
	
	\path[fill=black,line width=0.056pt] (122.0770,113.0030) node[above right]
	(text7400) {$\Omega_{t_2}$};
	
	\path[draw=black,dash pattern=on 0.45pt off 1.45pt,line join=miter,line
	cap=butt,miter limit=4.00,line width=0.56pt] (159.6949,50.9420) --
	(159.9981,63.4349);
\end{tikzpicture}
	\caption{Example of a weak field for a functions $u$ in a domain $\Omega$.}
	\label{Fig:Dibujo_viscosity}
\end{figure}


Finally, 
for the smaller class of convex Lagrangians,
a simple variational proof of our result can be given without using the calibration argument.
It will be explained in Remark~\ref{remark:barron}.
Here, in addition to the ellipticity condition~\eqref{Intro_lag:convex},
one needs to assume that the function $(a,b) \mapsto \lag(x,y,a,b)$ is convex;
in this regard see the first comments in Subsection~\ref{subsection:MainResults} and footnote~\ref{Footnote:Convex}.
The proof is a nonlocal counterpart of the one by Barron and Jensen~\cite{BarronJensen}.

\subsection{Outline of the article}
Section~\ref{section:nonlocal} contains 
the proofs of Theorem~\ref{thm:nonlocal:calibration} and Corollary~\ref{cor:applications}.
In Section~\ref{section:compound} we explain how to combine the local and nonlocal theory to obtain calibrations for mixed energy functionals.
In Section~\ref{section:viscosity} we apply the calibration formalism to the theory of viscosity solutions.
First, we prove the fully rigorous results for the fractional Laplacian (Theorems~\ref{thm:energy:comparison:frac} and \ref{thm:viscosity:frac}) with all details in regularity and integrability issues. 
Then, we show Theorems~\ref{thm:viscosity} and~\ref{thm:energy:comparison} (contained, respectively, in the more general 
Theorems~\ref{thm:gen:visc:min} and~\ref{thm:energy:comparison_combined}).

Appendix~\ref{section:comparison} 
contains an alternative proof, via a strong comparison principle, of the minimality for functions embedded in a field of extremals, under the additional assumption that an existence and regularity theorem for minimizers holds.
Finally, in Appendix~\ref{section:total:variation} we apply our calibration to the nonlocal total variation functional,
relating it with the calibration for the nonlocal perimeter constructed by the first author in \cite{Cabre-Calibration}.

\section{The calibration for general nonlocal functionals}
\label{section:nonlocal}

Having obtained a calibration for the semilinear problem involving the fractional Laplacian in our previous work~\cite{CabreErnetaFelipeNavarro-Calibration1}, we are now interested in extending this construction to a general class of nonlocal functionals.
In this way, we plan to obtain a similar picture to that of the general local theory treated in \cite[Section~3]{CabreErnetaFelipeNavarro-Calibration1}.
We find a functional $\calib$ that, at the formal level, is a calibration for the nonlocal energy functional~$\energy$. 
We said at the formal level since the appropriate regularity assumptions on the field of extremals will depend on the concrete given functional~$\energy$ and its associated nonlocal problem.

Consider the nonlocal energy functional $\energy$ of the form~\eqref{Intro_gen:energy}. 
Since $\domain$ is invariant with respect to the reflection $(x, y)\mapsto (y, x)$, we may assume without loss of generality that the Lagrangian $\lag$ is \emph{pairwise symmetric},\footnote{Here we follow the terminology of~\cite{Elbau}.} that is,
\begin{equation}\label{symmetry}
	\lag(y, x,b, a)= \lag(x, y, a, b) \quad \text{ for all } (x, y)\in \domain \text{ and } (a, b) \in \R^2.
\end{equation}
In particular, from the pairwise symmetry it follows that
\begin{equation}\label{der:symmetry}
	\partial_{b} \lag(x, y, \widetilde{a}, \widetilde{b}) = \partial_{a} \lag(y, x, \widetilde{b}, \widetilde{a}) \quad \text{ for all } (x, y)\in \domain \text{ and } (\widetilde{a}, \widetilde{b}) \in \R^2.
\end{equation}

The first variation of $\energy$ at $u$ in the direction of $\eta \in C_c^{\infty}(\R^n)$ (notice that $\eta$ is not necessarily supported in $\Omega$) is given by
\[
\begin{split}
	&\dfrac{\d}{\d \varepsilon} \energy(u + \varepsilon \eta)\Big|_{\varepsilon = 0} \\
	&= \dfrac{1}{2}\iint_{\domain}\!\! \partial_{a} \lag(x, y, u(x), u(y)) \, \eta(x) \d x \d y + \dfrac{1}{2}\iint_{\domain}\!\! \partial_{b} \lag(x, y, u(x), u(y)) \, \eta(y) \d x \d y\\
	&= \dfrac{1}{2}\iint_{\domain}\!\! \partial_{a} \lag(x, y, u(x), u(y)) \, \eta(x) \d x \d y + \dfrac{1}{2}\iint_{\domain}\!\! \partial_{b} \lag(y, x, u(y), u(x)) \, \eta(x) \d x \d y\\
	&= \iint_{\domain}\!\! \partial_{a} \lag(x, y, u(x), u(y)) \, \eta(x) \d x \d y,
\end{split}
\]
where we have used the symmetry of $\domain$ and the identity~\eqref{der:symmetry}.

Writing the domain $\domain$ as the disjoint union $\domain = (\Omega \times \R^{n}) \cup (\Omega^{c}\times \Omega)$, we can split the last integral to obtain
\begin{equation}\label{first:var}
	\begin{split}
		\dfrac{\d}{\d \varepsilon} \energy(u + \varepsilon \eta)\Big|_{\varepsilon = 0} 
		&= \int_{\Omega} \op (u)(x) \, \eta(x) \d x  + \int_{\Omega^{c}} \neumann(u)(x) \, \eta(x) \d x,
	\end{split}
\end{equation}
where we have introduced the nonlinear operators
\begin{equation*}
\op (u)(x) := \int_{\R^{n}} \partial_{a} \lag(x, y, u(x), u(y))\d y,
\end{equation*}
and
\begin{equation*}
\neumann (u)(x) := \int_{\Omega} \partial_{a} \lag (x, y, u(x), u(y)) \d y.
\end{equation*}
Consistent with the terminology in~\cite{CabreErnetaFelipeNavarro-Calibration1},
we refer to $\op$ as the \emph{Euler-Lagrange operator} associated to $\energy$, while $\neumann$ is its associated nonlocal \emph{Neumann operator}.

Since we are interested in minimization problems with respect to functions with the same exterior data, we only consider variations $\eta$ that are compactly supported in $\Omega$.
Thus, an extremal $u$ of $\energy$ will satisfy the Euler-Lagrange equation
\begin{equation}\label{euler:lagrange}
	\op (u) = 0 \quad \text{ in } \Omega.
\end{equation}

Given an interval $I \subset \R$, 
let $u^t\colon \R^n \to \R$ be a field in $\R^n$ (in the sense of Definition~\ref{def:foliation}), with $t \in I$, which covers the region
\[
\region := \{(x, \lambda)\in \R^{n}\times \R \,\colon\, \lambda = u^{t}(x) \text{ for some } t \in I\}.
\]
Let us also consider the class of admissible functions
\begin{equation}
\label{def:admissible}
\admissible := \{w \colon \R^{n} \to \R \,\colon\, w \text{ is sufficiently regular for } \lag \text{ and } {\rm graph }\, w \subset \region\},
\end{equation}
where ``sufficiently regular'' refers to the following issue.
Since we are not making any growth or structure assumption on $\lag$, the class of functions $w$ for which $\energy(w)$ makes sense must be chosen according to each nonlocal functional under investigation.
This will be the functions considered in $\admissible$, which may contain further regularity restrictions so that the operators $\op$ and $\neumann$, as well as all the integrals in the proofs, are well defined.

Let $t_0 \in I$. Our goal is to construct a calibration for $\energy$ and $u^{t_0}$.
We define the functional $\calib$ on $\admissible$ by
\begin{equation}\label{gen:calib}
	\calib(w) := \iint_{\domain} \int_{u^{t_0}(x)}^{w(x)} \partial_{a} \lag (x, y, u^{t}(x), u^{t}(y))\big|_{t = t(x, \lambda)}\d \lambda \d x \d y + \energy(u^{t_0}).
\end{equation}

By the above considerations and splitting the domain into $\domain = (\Omega \times \R^{n}) \cup (\Omega^{c}\times \Omega)$, we can rewrite~\eqref{gen:calib} as
\begin{equation}\label{gen:calib:sym}
	\begin{split}
		\calib(w) &= \int_{\Omega} \int_{u^{t_0}(x)}^{w(x)} \op(u^{t})(x)\big|_{t = t(x, \lambda)} \d \lambda \d x\\
		&\quad \quad \quad+\int_{\Omega^{c}} \int_{u^{t_0}(x)}^{w(x)} \neumann(u^{t})(x)\big|_{t = t(x, \lambda)} \d \lambda \d x + \energy(u^{t_0}).
	\end{split}
\end{equation}
Notice that~\eqref{gen:calib:sym} is the ``canonical'' nonlocal analogue of identity \eqref{Intro_calib:loc:2}, and thus of the classical local Weierstrass calibration~$\calibloc$; see Theorem~3.1 in \cite{CabreErnetaFelipeNavarro-Calibration1} for more details.

Next we show that if the field $\{u^t\}_{t \in I}$ is made up of supersolutions above $u^{t_0}$ and subsolutions below $u^{t_0}$, then $u^{t_0}$ minimizes $\calib$ among functions in $\admissible$ with the same exterior data.
Furthermore, if \emph{all} the functions $u^{t}$ satisfy the Euler-Lagrange equation (i.e., $u^t$ is a field of extremals), then $\calib$ is a null-Lagrangian and its value depends only on the exterior datum.
The following result (properties \ref{def:calib:1} and \ref{def:calib:1:prime} of the calibration) follows readily from expression~\eqref{gen:calib:sym} for $\calib.$ Note that here we do not need to assume the ellipticity of~$\lag$.
\begin{proposition}\label{property:1}
	Given an interval $I \subset \R$,
	a bounded domain $\Omega \subset \R^n$,
	and a pairwise symmetric function $\lag = \lag(x, y,a, b)$ in the sense of~\eqref{pairwise:symmetric},
	let $\{u^{t}\}_{t \in I}$ be a field in~$\R^n$ (in the sense of Definition \ref{def:foliation}) which is sufficiently regular for $\lag$.
	Assume that, for $t_0 \in I$, the leaves satisfy the inequalities
	\begin{equation}\label{sup:sub}
		\begin{split}
			\op(u^{t}) \geq 0 \quad \text{ in } \Omega 
			\quad
			\text{ for } t \geq t_0, \ \text{ and }
			\\
			\op(u^{t}) \leq 0 \quad \text{ in } \Omega \quad \text{ for } t \leq t_0. \ \ \ \ \ \ \ \ 
		\end{split}
	\end{equation}
	Consider the set of admissible functions $\admissible$ defined in \eqref{def:admissible}.
	
	Then, for all $w$ in $\admissible$ such that $w \equiv u^{t_0}$ in $\Omega^{c}$, 
	the functional $\calib$ defined in \eqref{gen:calib} satisfies
	\[
	\calib(u^{t_0})\leq \calib(w).
	\]
	
	Assume in addition that the leaves satisfy the Euler-Lagrange equation~\eqref{euler:lagrange}, that is,
	\begin{equation}\label{extremal}
		\op(u^{t}) = 0 \quad \text{ in } \Omega \quad \text{ for all } t \in I.
	\end{equation}
	Then, for all $w$ in $\admissible$ such that $w \equiv u^{t_0}$ in $\Omega^{c}$, we have
	\[
	\calib(u^{t_0}) = \calib(w).
	\]
\end{proposition}

\begin{proof}
	First, notice that $\calib(u^{t_0}) = \energy(u^{t_0})$.
	Since $w \equiv u^{t_0}$ in $\Omega^c$, by~\eqref{gen:calib:sym} we have
	\[
	\calib(w) - \calib(u^{t_0}) = \int_{\Omega}\int_{u^{t_0}(x)}^{w(x)} \op(u^{t})\big|_{t = t(x, \lambda)} \d \lambda \d x.
	\]
	Assuming~\eqref{sup:sub}, it suffices to show that for all $x \in \Omega$ we have
	\begin{equation}\label{ineq:aux}
		\int_{u^{t_0}(x)}^{w(x)} \op(u^{t})\big|_{t = t(x, \lambda)} \d \lambda \geq 0.
	\end{equation}
	If $w(x) \geq u^{t_0}(x)$, then, using that the functions $\{u^{t}\}_{t \in I}$ are increasing in $t$, we have $t(x,\lambda)\geq t_0$ for $\lambda \in [u^{t_0}(x), w(x)]$. 
Hence, 
by assumption~\eqref{sup:sub}, 
$\op(u^{t})\big|_{t = t(x, \lambda)} \geq 0$ and ~\eqref{ineq:aux} follows in this case.
The case $w(x)\leq u^{t_0}(x)$ is treated similarly.

If we further assume~\eqref{extremal}, then the integral in~\eqref{ineq:aux} vanishes and the claim follows.
\end{proof}

The functional $\calib$ can be rewritten in the following alternative form that we will use to verify the remaining calibration properties~\ref{def:calib:2} and~\ref{def:calib:3}.
\begin{lemma}
	\label{lemma:alt:expression}
	Given an interval $I \subset \R$,
	a bounded domain $\Omega \subset \R^n$,
	and a  pairwise symmetric function $\lag = \lag(x, y,a, b)$ in the sense of~\eqref{symmetry},
let $\{u^{t}\}_{t \in I}$ be a field in~$\R^n$ (in the sense of Definition \ref{def:foliation}) which is sufficiently regular for $\lag$.
	Consider the set of admissible functions $\admissible$ defined in~\eqref{def:admissible}.
	
	Then,
	for all $w$ in $\admissible$,
	the functional $\calib$ defined in \eqref{gen:calib} satisfies
	\begin{equation}\label{alt:expression}
		\begin{split}
			\calib(w) &= \dfrac{1}{2}\iint_{\domain} \lag(x, y, w(x), u^{t}(y))\big|_{t = t(x, w(x))} \d x \d y \\
			&\quad\quad + \dfrac{1}{2}\iint_{\domain} \int_{t(x, w(x))}^{t(y, w(y))} \partial_{b} \lag(x, y, u^{t}(x), u^{t}(y)) \partial_t u^{t}(y) \d t \d x \d y.
		\end{split}
	\end{equation}
\end{lemma}
\begin{remark}
While the definition of the calibration in~\eqref{gen:calib} seemingly depends on a given $t_0$,
Lemma~\ref{lemma:alt:expression} shows that it is, in fact, independent of this choice.
\end{remark}
\begin{proof}
We first rewrite the integral term in the definition~\eqref{gen:calib} of $\calib$.
Applying the change of variables $\lambda \mapsto t$ with $u^{t}(x) = \lambda$ for each $x$, we have
	\[
	\begin{split}
		&\iint_{\domain} \int_{u^{t_0}(x)}^{w(x)}\partial_{a} \lag (x, y, u^{t}(x), u^{t}(y))\big|_{t = t(x, \lambda)}\d \lambda \d x \d y \\
		& \hspace{2cm} = \iint_{\domain} \int_{t_0}^{t(x, w(x))}  \partial_{a} \lag (x, y, u^{t}(x), u^{t}(y)) \partial_t u^{t}(x)\d t \d x \d y.
	\end{split}
	\]
	Symmetrizing this expression in $(x, y)$ and using~\eqref{der:symmetry}, we deduce	
	\begin{equation}\label{l:1}
		\begin{split}
			&\iint_{\domain} \int_{u^{t_0}(x)}^{w(x)}\partial_{a} \lag (x, y, u^{t}(x), u^{t}(y))\big|_{t = t(x, \lambda)}\d \lambda \d x \d y \\
			&\hspace{1.5cm}= \dfrac{1}{2}\iint_{\domain}\int_{t_0}^{t(x, w(x))} \partial_{a} \lag(x, y, u^{t}(x), u^{t}(y)) \partial_t u^{t}(x) \d t \d x \d y \\
			&\hspace{2.5cm}\quad +\dfrac{1}{2}\iint_{\domain}\int_{t_0}^{t(y, w(y))} \partial_{b}\lag(x, y, u^{t}(x), u^{t}(y)) \partial_t u^{t}(y) \d t \d x \d y.
		\end{split}
	\end{equation}
	Splitting the integral $\int_{t_0}^{t(y, w(y))} \cdot \d t$ in~\eqref{l:1} into $\int_{t_0}^{t(x, w(x))}\cdot \d t + \int_{t(x, w(x))}^{t(y, w(y))} \cdot \d t$, we obtain
	\begin{equation}\label{l:2}
		\begin{split}
			&\iint_{\domain} \int_{u^{t_0}(x)}^{w(x)} \partial_{a} \lag (x, y, u^{t}(x), u^{t}(y))\big|_{t = t(x, \lambda)}\d \lambda \d x \d y \\
			&\hspace{2cm}= \dfrac{1}{2}\iint_{\domain} \int_{t_0}^{t(x, w(x))} 
			\dfrac{\d}{\d t}\big\{ \lag(x, y, u^{t}(x), u^{t}(y)) \big\} 
			\d t \d x \d y\\
			&\hspace{2.5cm}\ \ \quad +
			\dfrac{1}{2}\iint_{\domain} \int_{t(x, w(x))}^{t(y, w(y))} \partial_{b} \lag (x, y, u^{t}(x), u^{t}(y)) \partial_t u^{t}(y) \d t \d x \d y.
		\end{split}
	\end{equation}
	Integrating the derivative with respect to $t$ in \eqref{l:2} and recalling, by definition of the leaf-parameter function, that $w(x) = u^{t(x,w(x))}(x)$, we have
	\begin{equation}\label{l:3}
		\begin{split}
			&\iint_{\domain} \int_{u^{t_0}(x)}^{w(x)} \partial_{a} \lag (x, y, u^{t}(x), u^{t}(y))\big|_{t = t(x, \lambda)}\d \lambda \d x \d y \\
			&\hspace{2cm}= \dfrac{1}{2}\iint_{\domain} \lag(x, y, w(x), u^{t(x, w(x))}(y)) \d x \d y - \energy(u^{t_0})\\
			&\hspace{2.5cm}\ \ \quad +\dfrac{1}{2}\iint_{\domain} \int_{t(x, w(x))}^{t(y, w(y))} \partial_{b} \lag (x, y, u^{t}(x), u^{t}(y)) \partial_t u^{t}(y) \d t \d x \d y.
		\end{split}
	\end{equation}
	Adding $\energy(u^{t_0})$ to both sides of~\eqref{l:3} now yields the claim. 
\end{proof}

In the next proposition we prove the calibration property~\ref{def:calib:2}. This follows directly from Lemma~\ref{lemma:alt:expression}. Here, ellipticity of~$\lag$ is still not needed.
\begin{proposition}\label{property:2}
	Given an interval $I \subset \R$,
	a bounded domain $\Omega \subset \R^n$,
	and a  pairwise symmetric function $\lag = \lag(x, y,a, b)$ in the sense of~\eqref{symmetry},
let $\{u^{t}\}_{t \in I}$ be a field in~$\R^n$ (in the sense of Definition \ref{def:foliation}) which is sufficiently regular for $\lag$.
	
	Then, for all $t \in I$, 
	the functional $\calib$ defined in \eqref{gen:calib} satisfies
	\[
	\calib(u^{t}) = \energy(u^{t}).
	\]
\end{proposition}
\begin{proof}
Let $t_0 \in I$.
Choosing $w = u^{t_0}$ in~\eqref{alt:expression}, since $t(x, w(x)) = t_0$ for all $x$, we have
\[
\calib(u^{t_0}) = \dfrac{1}{2}\iint_{\domain} \lag(x, y, u^{t_0}(x), u^{t_0}(y)) \d x \d y = \energy(u^{t_0}).
\]
Since $t_0 \in I$ was arbitrary, this proves the proposition.
\end{proof}

It remains to prove the last calibration property~\ref{def:calib:3}.
We will now need the natural ellipticity assumption~\eqref{Intro_lag:convex} on the Lagrangian $\lag$.
This condition is related to a strong comparison principle, as explained in Appendix~\ref{section:comparison}.

\begin{proposition}\label{property:3}
	Given an interval $I \subset \R$,
	a bounded domain $\Omega \subset \R^n$,
	and a  pairwise symmetric function $\lag = \lag(x, y,a, b)$ in the sense of~\eqref{symmetry},
let $\{u^{t}\}_{t \in I}$ be a field in~$\R^n$ (in the sense of Definition \ref{def:foliation}) which is sufficiently regular for $\lag$.
		Consider the set of admissible functions $\admissible$ defined in \eqref{def:admissible}.
	Assume that the ellipticity condition $\partial_{a b}^2 \lag \leq 0$ holds.
	
	Then, for all $w$ in $\admissible$,
	the functional $\calib$ defined in \eqref{gen:calib} satisfies
	\[
	\calib(w) \leq \energy(w).
	\]
\end{proposition}
\begin{proof}
	Computing the difference $\energy(w) - \calib(w)$, from the alternative expression \eqref{alt:expression} for $\calib$, we obtain 
	\begin{equation}\label{p3:1}
		\begin{split}
			&\energy(w) - \calib(w) \\
			&\hspace{1cm}= \dfrac{1}{2}\iint_{\domain} \big\{ \lag(x, y, w(x), w(y))- \lag(x, y, w(x), u^{t(x, w(x))}(y)) \big\} \d x \d y\\
			&\hspace{1.5cm}\quad -
			\dfrac{1}{2}\iint_{\domain} \int_{t(x, w(x))}^{t(y, w(y))} {\partial_{b} \lag}(x, y, u^{t}(x), u^{t}(y)) \partial_t u^{t}(y) \d t \d x \d y.
		\end{split}
	\end{equation}
	Recalling that $u^{t(y, w(y))}(y) = w(y)$, we can write the first integral on the right-hand side of~\eqref{p3:1} as
	\begin{equation}\label{p3:2}
		\begin{split}
			&\dfrac{1}{2}\iint_{\domain} \big\{\lag(x, y, w(x), w(y))- \lag(x, y, w(x), u^{t(x, w(x))}(y)) \big\} \d x \d y\\
			&\hspace{1.5cm}= \dfrac{1}{2}\iint_{\domain} \int_{t(x, w(x))}^{t(y, w(y))} \dfrac{\d}{\d t}\big\{\lag(x, y, w(x), u^{t}(y))\big\} \d t \d x \d y\\
			&\hspace{1.5cm}= \dfrac{1}{2}\iint_{\domain} \int_{t(x, w(x))}^{t(y, w(y))} \partial_{b} \lag(x, y, w(x), u^{t}(y))\partial_t u^{t}(y) \d t \d x \d y.
		\end{split}
	\end{equation}
	Plugging~\eqref{p3:2} into~\eqref{p3:1}, we see that
	\begin{align*}
		&\energy(w) - \calib(w) \\
		&\ = \iint_{\domain} \! \int_{t(x, w(x))}^{t(y, w(y))}
		\!\!\!\big\{
		\partial_{b} \lag(x, y, w(x), u^{t}(y))
		- \partial_{b} \lag(x, y, u^{t}(x), u^{t}(y)) 
		\big\} 
		\, \partial_t u^{t}(y) \d t \d x \d y.
	\end{align*}
	Thus, it suffices to show that
	\begin{equation}\label{p3:claim}
		\int_{t(x, w(x))}^{t(y, w(y))}\big\{
		\partial_{b} \lag(x, y, w(x), u^{t}(y))
		-\partial_{b} \lag(x, y, u^{t}(x), u^{t}(y))  
		\big \} \, \partial_t u^{t}(y) \d t \geq 0,
	\end{equation}
	for all $(x, y) \in \domain$.
	
	Let $(x, y)\in \domain$ and assume first that $t(x, w(x)) \leq t(y, w(y))$.
By monotonicity of the leaves $u^t$ in $I$, for $t \in [t(x,w(x)), t(y, w(y))]$ we have
	\[
	w(x) = u^{t(x, w(x))}(x) \leq u^{t}(x),
	\]
	and by ellipticity
	\[
	\partial_{b} \lag(x, y, w(x), u^{t}(y))\partial_t u^{t}(y) \geq \partial_{b} \lag(x, y, u^{t}(x), u^{t}(y))\partial_t u^{t}(y).
	\]
	Whence, \eqref{p3:claim} follows.
	The case $t(x, w(x)) \geq t(y, w(y))$ is treated similarly.
\end{proof}

Now, combining Propositions~\ref{property:1},~\ref{property:2}, and~\ref{property:3}, we easily conclude Theorem~\ref{thm:nonlocal:calibration}.

\begin{proof}[Proof of Theorem~\ref{thm:nonlocal:calibration}]
	(a) Property~\ref{def:calib:2} follows from Proposition~\ref{property:2} and property~\ref{def:calib:3} follows from Proposition~\ref{property:3}. 
	
	(b) This follows from the first part of Proposition~\ref{property:1}.
	
	(c) This follows from the second part of Proposition~\ref{property:1}.
\end{proof}

{
	Now that we have built a calibration for fields of extremals, we can show Corollary~\ref{cor:applications} on the minimality of monotone solutions to translation invariant equations.
}

\begin{proof}[Proof of Corollary~\ref{cor:applications}]
	For each $t \in \R$ we define $u^{t}(x) := u(x', x_n + t)$, where $x = (x', x_n) \in \R^{n-1}\times \R$.
	By the monotonicity~\eqref{monotone} of $u$
	and by translation invariance of the equation $\op(u) = 0$,
	it follows that the family $\{u^{t}\}_{t \in \R}$ is a field of extremals in~$\R^n$ in the sense of Definition~\ref{def:foliation}.
	Hence, Theorem~\ref{thm:nonlocal:calibration} yields the minimality of each $u^{t}$ among competitors $w$ with $w\equiv u$ in $\Omega^c$ and satisfying the assumption
	$$\lim_{\tau \to -\infty} u(x', \tau) < w(x', x_n) < \lim_{\tau \to +\infty} u(x', \tau) \ \ \text{ for all } \ \ x=(x',x_n)\in\Omega.$$
	
	Finally, we can relax the previous strict inequalities by considering the competitor 
	$(1-\varepsilon) w + \varepsilon u$ and letting~$\varepsilon~\to~0$.
	In this way we recover the condition in the statement of Corollary~\ref{cor:applications} where the inequalities are not strict.
\end{proof}

\section{The calibration for functionals involving both local and nonlocal terms}
\label{section:compound}

The results derived in Section~\ref{section:nonlocal} may be combined with the classical local ones to yield a theory that applies to functionals involving both local and nonlocal interactions. 
These functionals appear when dealing with symmetric Lévy processes, where the infinitesimal generators 
are given by the sum of a second order differential operator and an integro-differential one. Recently, mixed functionals have attracted great attention from different points of view; see~\cite{SuValdinociWeiZhang, MaioneMugnaiVecchi} 
and references therein.

The mixed energy\footnote{
Here in the notation we use the subscript M, which stands for ``Mixed''.
}
\begin{equation}
\label{def:mixed}
	\begin{split}
		\energytot(w) &:= \energy(w) + \energyloc(w)
		\\
		&= \dfrac{1}{2}\iint_{\domain} \lag(x, y, w(x), w(y)) \d x \d y + \int_{\Omega} \lagloc(x, w(x), \nabla w(x)) \d x,
	\end{split}
\end{equation}
admits a calibrating functional 
\begin{equation*}
	\calibtot(w) := \calib(w) + \calibloc(w),
\end{equation*}
where 
$\energyloc$, $\lagloc$, $\calibloc$, and $\energy$, $\lag$, $\calib$ are defined as in the Introduction.
By combining identities~\eqref{Intro_calib:loc:2} and~\eqref{gen:calib:sym}, the functional $\calibtot$ may be written equivalently as
\begin{equation*}
	\begin{split}
		\calibtot(w) &= \int_{\Omega} \int_{u^{t_0}(x)}^{w(x)} \optot(u^{t})\big|_{t = t(x, \lambda)} \d \lambda \d x\\
		&\quad \ \ 
		+\int_{\Omega^{c}} \int_{u^{t_0}(x)}^{w(x)} \neumann(u^{t})\big|_{t = t(x, \lambda)} \d \lambda \d x + \int_{\partial \Omega} \int_{u^{t_0}(x)}^{v(x)}  \neumannloc(u^{t}) \big|_{t = t(x,\lambda)} \d \lambda \d \mathcal{H}^{n-1}(x)\\
		&\quad \ \ +\energytot(u^{t_0}),
	\end{split}
\end{equation*}
where the Euler-Lagrange operator of the mixed problem is
\begin{equation*}
	\optot(w) := \op(w) + \oploc(w),
\end{equation*}
and 
$\oploc$, $\neumannloc$, $\op$, and $\neumann$ are the operators introduced above.

Since $\calibtot$ shares the same structure as $\calibloc$ and $\calib$, a straightforward adaptation of the proofs in the sections above shows that $\calibtot$ satisfies all three calibration properties.
We mention that property~\ref{def:calib:3} requires both the local and nonlocal ellipticity conditions, that is, one must assume that both
$$\partial^2_{qq} \lagloc(x, \lambda, q) \geq 0 \ \  \text{and} \ \ \partial^2_{a b} \lag(x, y, a, b) \leq 0$$
hold.

As an application of this theory, we can prove the analogue of Corollary~\ref{cor:applications} for mixed functionals.
Namely, if $\optot$ is translation invariant, i.e., $\optot(u(\cdot+y))(x) = \optot(u)(x+y)$ for all $x$ and $y$  in $\R^{n}$, then monotone solutions are minimizers among functions lying between the limits of the solution in the direction of monotonicity.
The proof is identical to the one of Corollary~\ref{cor:applications}.

\begin{remark}
Mixed energies appear in some relevant frameworks, as listed next. Note however that these minimization problems include constraints.
Thus, one cannot directly apply the calibration theory developed above, since constrained minimizers need not be minimizers of the original functional and no foliation of extremals is expected.
As examples of such frameworks, we mention the theory of aggregation equations~\cite{CarrilloCraigYao},
certain problems from astrophysics~\cite{LionsMinimizationL1}, the Thomas-Fermi theory~\cite{BenilanBrezis}, the Choquard-Pekar model~\cite{Lieb}, as well as the problem of finding the best constant in the Sobolev inequality~\cite{Talenti}.
\end{remark}

\section{Application to the viscosity theory}
\label{section:viscosity}

For the application of calibrations to prove that minimizers are viscosity solutions, we need to consider more general fields, namely, 
those which are not necessarily increasing in a bounded domain $\Omega \subset \R^n$, but only nondecreasing.
We need to consider the situation in which different leaves coincide in certain subsets of $\Omega$; see Figure \ref{Fig:Dibujo_viscosity} in the Introduction.
Such a field will appear when sliding a touching test function and truncating it with the minimizer. 

Due to the great generality of the nonlocal elliptic functionals considered in this work, 
general statements about them can only be given in a formal sense, as the ones presented in the previous sections.
Nevertheless, as mentioned above, such statements can be made fully rigorous under natural growth and regularity assumptions on the Lagrangian $\lag$.

In the present section, 
we first obtain rigorous results for 
semilinear equations involving the fractional Laplacian (Subsections~\ref{subsection:energyComparisonFractional} and~\ref{subsection:viscosityFractional})
and then present formal results concerning general nonlocal functionals (Subsection~\ref{subsection:viscosityGeneral}). 
Subsection~\ref{subsection:energyComparisonFractional} is devoted to the proof of a fractional energy comparison result for solutions embedded in a weak field.
This is a new result that is obtained by a calibration argument.
We then apply the energy comparison in Subsection~\ref{subsection:viscosityFractional} to conclude that minimizers of fractional functionals are viscosity solutions.
Finally, in Subsection~\ref{subsection:viscosityGeneral}, we extend the previous results to formal statements for 
general nonlocal functionals, including mixed functionals involving both local and nonlocal terms.

\subsection{An energy comparison result for the fractional Laplacian} \label{subsection:energyComparisonFractional}
We will follow the notation from our previous work~\cite{CabreErnetaFelipeNavarro-Calibration1},
recalled briefly here.
For $s \in (0,1)$ we consider 
the space of weighted $L^1$ functions
\[
L^1_{s}(\R^n) = \left\{ u \in L^1_{\rm loc}(\R^n) \colon \|u\|_{L^1_{s}(\R^n)} = \int_{\R^n} \frac{|u(y)|}{1 + |y|^{n+2s}} \d y < +\infty \right\}.
\]
If $u \in L^1_{s}(\R^n)$ is $C^2$ in a neighborhood of a point $x \in \R^n$, then the fractional Laplacian
\[
(-\Delta)^s u(x) = c_{n,s} \PV \int_{\R^n} \frac{u(x) - u(y)}{|x-y|^{n+2s}} \d y
\]
introduced above is well defined.

Given $s \in (0,1)$ and a bounded domain $\Omega\subset \R^{n}$,
let $u$ belong to $C(\overline{\Omega}) \cap \lsn$.
Assume that we are given functions $u^{t} \colon \R^{n} \to \R$,
with $t \in [0, \tmax]$ and nondecreasing in~$t$, 
and such that $u^0 = u$.
When the family $\{u^{t}\}_{t \in [0,\tmax]}$ satisfies appropriate regularity assumptions, we will be able to construct a calibration involving the weak field.

Consider the region 
\[
\region = \big\{(x,\lambda)\in \Omega\times \R : u^0(x)<\lambda \leq u^\tmax(x) \big\},
\]
as well as the sections
\[
\Omega_t := \{x \in \Omega \colon u^t(x) > u^0(x) 
\} \quad \text{ for each } t \in (0, \tmax],
\]
which will be increasing sets in $t$, and 
\[
I_x := \{t \in (0, \tmax] \colon u^t(x) > u^0(x)
\} 
\quad \text{ for each } x \in \Omega.
\]

\begin{definition}\label{def:visc:field}
Given $s \in (0,1)$ and a bounded domain $\Omega\subset \R^{n}$,
let $u$ 
$\in C(\overline{\Omega})~\cap~\lsn$.
	A family $\{u^{t}\}_{t \in [0,\tmax]}$ of functions $u^{t}\colon \R^n \to \R$ is said to be a 
	\emph{weak field for $u$ 
		in $\Omega$} (see Figure~\ref{Fig:Dibujo_viscosity})
	if
	the following conditions are satisfied:
	\begin{enumerate}[label= (\roman*)]
		\item\label{def:visc:field:1}
		$u^0 = u$.
		\item\label{def:visc:field:2}
		$u^t = u$ a.e. in $\Omega^c$, for all $t \in [0,\tmax]$.
		\item\label{def:visc:field:3} The function $(x,t) \mapsto u^t(x)$ is continuous in $\overline{\Omega}\times [0,\tmax]$.
		\item\label{def:visc:field:4} 
		For each $x \in \Omega$, the function $t \mapsto u^t(x)$ is $C^1(I_x)$ and increasing in $\overline{I_x}$.
		Moreover, there exists a constant $C_0 > 0$ such that
		\[
		0 \leq \partial_{t}u^{t}(x) \leq C_0 \quad \text{for all } x \in \Omega \text{ and } t \in I_x.
		\]
	\end{enumerate}	
	Moreover, the weak field is \emph{regular by below}, if the following regularity condition holds:
	\begin{enumerate}[label= (\roman*)]
		\setcounter{enumi}{4}
		\item\label{def:visc:field:5}
		The functions $\{u^t\}_{t \in (0,\tmax)}$ are uniformly $C^{1,1}$ by below in $\Omega_t$, uniformly in $t$,
		in the following sense.
		There exist a constant $C_0 > 0$ and a bounded domain $N \subset \R^n$, with $\overline{\Omega} \subset N$, such that, 
		for each $t \in (0,\tmax)$ and $x \in \Omega_t$,
		there is a function $\psi \in C^2(N)$ touching $u^{t}$ by below in $N$ at $x$, that is,
		$\psi(x) = u^{t}(x)$ and $\psi \leq u^{t}$ a.e. in $N$,
		satisfying
		\[
		D^2 \psi \geq -C_{0} \quad \text{ in }N.
		\]
	\end{enumerate}
\end{definition}

\begin{remark}
	The more technical assumption in Definition~\ref{def:visc:field}, condition \ref{def:visc:field:5}, is needed for the calibration of the fractional Laplacian to be well defined.
	An important consequence of \ref{def:visc:field:5} is that the fractional Laplacian $(-\Delta)^s u^{t}(x)$ is bounded by above uniformly in $t \in (0,\tmax)$ and $x \in \Omega_t$.

In our previous work~\cite{CabreErnetaFelipeNavarro-Calibration1}, we considered the more regular class of $C^2$ fields.
For these fields, the fractional Laplacian is continuous in a neighborhood of $\overline{\Omega}$ and hence the calibration takes finite values on functions with finite energy.
By contrast, in our application to the viscosity theory we will construct a weak field by sliding a $C^2$ function and taking the maximum with the critical point $u$.
This naturally yields a field which is only $C^{1,1}$ by below.
The calibration obtained from this weak field is bounded by above, but not necessarily by below.
In particular, the calibration could be $-\infty$ at some admissible functions (even at some leaves of the weak field).
However, the bound by above suffices to carry out the argument.
\end{remark}

Recall that, given $s \in (0,1)$  and $F \in C^1(\R)$, we have
\begin{equation*}
\energyfrac(w)
=
\dfrac{c_{n,s}}{4}\iint_{\domain} \dfrac{|w(x)-w(y)|^2}{|x-y|^{n+2s}} \d x \d y - \int_{\Omega} F(w(x)) \d x.
\end{equation*}
We also write
\[
\gagliardo(w) = \dfrac{c_{n,s}}{4}\iint_{\domain} \dfrac{|w(x)-w(y)|^2}{|x-y|^{n+2s}} \d x \d y,
\]
following the notation introduced in our previous work~\cite{CabreErnetaFelipeNavarro-Calibration1}.
The weak field allows to compare the energies of $u$ and any leaf $u^t$ via the following theorem:

\begin{theorem}
	\label{thm:energy:comparison:frac}
	Given a bounded domain $\Omega \subset \R^n$, $s \in (0,1)$, and $u \in C(\overline{\Omega})~\cap~\lsn$,
	let $\{u^{t}\}_{t \in [0,\tmax]}$ be a weak field for $u$
	 in $\Omega$
	which is regular by below
	in the sense of Definition~\ref{def:visc:field}.
	
	Then, if $\gagliardo(u) < +\infty$, given $F \in C^1(\R)$  we have
	\begin{equation}
	\label{last:ineq}
		\energyfrac(u^{T})\leq \energyfrac(u) + \int_{\Omega \cap \{u^{T}>u\}} \int_{u(x)}^{u^{T}(x)} \big( (-\Delta)^s u^{t}(x) - F'(u^{t}(x))\big) \big|_{t = t(x, \lambda)} \d \lambda \d x.
	\end{equation}
\end{theorem}
Notice that the right-hand side in~\eqref{last:ineq} is well defined.
Indeed, while the second term could be $-\infty$,
the first one is finite by assumption.

\begin{proof}[Proof of Theorem~\ref{thm:energy:comparison:frac}]
First, note that for each $(x, \lambda) \in \region$, there exists a unique leaf-parameter function $t = t(x,\lambda) \in (0, T]$ such that $u^{t(x,\lambda)}(x) = \lambda$.  
Existence follows from property~\ref{def:visc:field:3}, since $u^{0}(x) < \lambda \leq u^{\tmax}(x)$, while uniqueness is guaranteed by~\ref{def:visc:field:4}.

We can naturally extend this definition to a larger domain.
For $x \in \Omega_T$, we define
\[
t(x,u^{0}(x)) := \lim_{\lambda \downarrow u^{0}(x)} t(x,\lambda).
\]
This extension makes the function $\lambda \mapsto t(x, \lambda)$ continuous up to the boundary on the interval $[u^{0}(x), u^{T}(x)]$.
For $x \in \Omega_T^{c}$, we simply set
\[
t(x, u^{0}(x)) = t(x,u^T(x)) := T.
\]
As a consequence of the above definition, we have the important identity
\[
u^{t(x, u^\tau(x))}(x) = u^\tau(x) \quad \text{for all } x \in \R^n \text{ and } \tau \in [0, \tmax].
\]
Here, the case $\tau \in (0, \tmax]$ is immediate, while $\tau = 0$ follows from the continuity of both the field and the leaf-parameter function.
	
Now, we proceed as in the proof of the calibration properties in our previous work~\cite{CabreErnetaFelipeNavarro-Calibration1}.
The idea is to consider the analogue of the fractional calibration $\calibfrac$ for $\energyfrac$ and $u^\tmax$ introduced in~\cite{CabreErnetaFelipeNavarro-Calibration1}, now for the weak field, and to use property~\ref{def:calib:3}, i.e., $\calibfrac(u) \leq \energyfrac(u)$.
This will lead to the desired energy comparison. 
The details go as follows.

As in~\cite{CabreErnetaFelipeNavarro-Calibration1}, for $\varepsilon> 0$, we consider the kernel $K_{\varepsilon } = c_{n,s}|\cdot|^{-n-2s}  \mathds{1}_{\R^{n}\setminus B_{\varepsilon}}$ and the truncated fractional Laplacian $(-\Delta)^s_{\varepsilon} \varphi(x) = \int_{\R^n} (\varphi(x) - \varphi(y))K_{\varepsilon}(x-y) \d y$.

Since $u$ and $u^\tmax$ differ only in the region $\Omega_{\tmax} \subset \Omega$, 
it suffices to consider the functionals $\gagliardo(u)$ and $\gagliardo(u^\tmax)$ 
defined in $Q(\Omega_{\tmax})$ instead of in the larger set 
\[
\domain = Q(\Omega_{\tmax}) \cup ((\Omega \setminus \Omega_{\tmax})\times \Omega_{\tmax}^{c}) \cup (\Omega^c \times (\Omega \setminus \Omega_{\tmax})).
\]
By a slight modification of the proof of Lemma~4.4 in~\cite{CabreErnetaFelipeNavarro-Calibration1},
(namely, making the change $\Omega \to \Omega_T$, $u^{t_0} \mapsto u^\tmax$, and $w \mapsto u$ in the last statement of that proof) 
we have the identity
\begin{equation}
\label{sir1}
\begin{split}
& \int_{\Omega_{\tmax}} \int_{u^\tmax(x)}^{u(x)} (-\Delta)^s_{\varepsilon} u^{t}(x) \big|_{t = t(x, \lambda)}\d x + \frac{1}{4}\iint_{Q(\Omega_{\tmax})} |u^\tmax(x)-u^\tmax(y)|^2  K_{\varepsilon}(x-y) \d x \d y\\
& \quad = - \frac{1}{2}\iint_{Q(\Omega_{\tmax})} \d x \d y K_{\varepsilon}(x-y) \int^{t(y,u(y))}_{t(x,u(x))}  (u^{t}(x)-u^{t}(y)) \,\partial_t{u}^t(y) \d t\\
&\quad \quad \quad \quad +\frac{1}{4}\iint_{Q(\Omega_{\tmax})}  |u(x) -u^{t(x,u(x))}(y)|^2 K_{\varepsilon}(x-y) \d x \d y.
\end{split}
\end{equation}
Moreover, since $u^{t}$ is nondecreasing in $t$, the first term in the right-hand side of \eqref{sir1} can be bounded by
\begin{equation}
\label{sir2}
\begin{split}
& - \frac{1}{2}\iint_{Q(\Omega_{\tmax})} \d x \d y K_{\varepsilon}(x-y) \int^{t(y,u(y))}_{t(x,u(x))}  (u^{t}(x)-u^{t}(y))\, \partial_t{u}^t(y) \d t \\
& \quad  \quad  \quad  \quad 
\leq - \frac{1}{2}\iint_{Q(\Omega_{\tmax})} \d x \d y K_{\varepsilon}(x-y) \int^{t(y,u(y))}_{t(x,u(x))}  (u(x)-u^{t}(y))\, \partial_t{u}^t(y) \d t \\
& \quad  \quad  \quad  \quad 
= \frac{1}{4}\iint_{Q(\Omega_{\tmax})} \d x \d y K_{\varepsilon}(x-y) \big(|u(x)-u(y)|^2 - |u(x)-u^{t(x,u(x))}(y)|^2\big),
\end{split}
\end{equation}
where in the last line we have used $- 2(u(x) - u^t(y)) \partial_t u^t(y) = \frac{\d}{\d t} |u(x) - u^t(y)|^2$.
Hence, 
denoting $\gagliardoeps(u) := \frac{1}{2}\iint_{\domain} |u(x)-u(y)|^2 K_{\varepsilon}(x-y) \d x \d y$,
by combining \eqref{sir1} and \eqref{sir2} we deduce
\begin{equation}
\label{sir3}
\begin{split}
\gagliardoeps(u) - \gagliardoeps(u^\tmax) &= \frac{1}{4}\iint_{Q(\Omega_{\tmax})} \big( |u(x)-u(y)|^2- |u^\tmax(x)-u^\tmax(y)|^2\big) K_{\varepsilon}(x-y) \d x \d y \\
&\geq \int_{\Omega_{\tmax}} \int_{u^\tmax(x)}^{u(x)} (-\Delta)^s_{\varepsilon} u^{t}(x) \big|_{t = t(x, \lambda)} \d \lambda \d x \\
&= \int_{\Omega_{\tmax}} \int_{u(x)}^{u^\tmax(x)} -(-\Delta)^s_{\varepsilon} u^{t}(x) \big|_{t = t(x, \lambda)} \d \lambda \d x.
\end{split}
\end{equation}

Finally, thanks to property \ref{def:visc:field:5}, the functions $-(-\Delta)^s_{\varepsilon} u^t(x)\big|_{t = t(x,\lambda)}$ are bounded by below 
in $\{(x, \lambda) \in \Omega_T \times \R \colon u(x) < \lambda < u^\tmax(x)\}$, uniformly in $\varepsilon$.
By Fatou's lemma we can pass to the limit as $\varepsilon \downarrow 0$ inside the integrals in \eqref{sir3} to obtain
	\[
	\gagliardo(u^\tmax) \leq \gagliardo(u) + \int_{\Omega_{\tmax}} \int_{u(x)}^{u^\tmax(x)} (-\Delta)^s u^{t}(x)\big|_{t = t(x,\lambda)} \d \lambda \d x.
	\]
Since $F(u(x)) - F(u^\tmax(x)) = \int_{u^\tmax(x)}^{u(x)} F'(\lambda) \d \lambda = \int_{u^\tmax(x)}^{u(x)} F'(u^{t}(x))\big|_{t = t(x, \lambda)} \d \lambda$,
adding the potential term now yields the result.
\end{proof}

\subsection{
Minimizers of functionals involving the Gagliardo seminorm are viscosity solutions
} \label{subsection:viscosityFractional}

We recall the definition of viscosity solution in the nonlocal setting.
The following is taken from Definition~2.2 in~\cite{CaffarelliSilvestreRegularity}:
\begin{definition}
\label{def:viscosity}
	Given bounded domain $\Omega \subset \R^n$, $s \in (0,1)$, and $F \in C^{1}(\R)$, we say that $u \in C(\Omega) \cap \lsn$ is a \emph{viscosity supersolution} of 
	the semilinear equation 
	\[
	(-\Delta)^s v = F'(v) \text{ in }  \Omega,
	\]
	if whenever the following happens
	\begin{itemize}
		\item $x_0$ is any point in $\Omega$
		\item $N$ is a neighborhood of $x_0$ in $\Omega$
		\item $\varphi$ is some $C^2$ function in $\overline{N}$
		\item $\varphi(x_0) = u(x_0)$
		\item $\varphi(x) < u(x)$ for every $x \in N \setminus\{x_0\}$,
	\end{itemize}
	then, the function
	\begin{equation} \label{extension-bar}
	\overline{\varphi}(x) := 
	\left\{
	\begin{array}{ll}
		\varphi(x) & \text{ for } x \in N,\\
		u(x) & \text{ for } x \in \R^n \setminus N,\\
	\end{array}
	\right.
	\end{equation}
	satisfies $(-\Delta)^s \overline{\varphi}(x_0) \geq F'(\overline{\varphi}(x_0))$.
\end{definition}
We also have the analogous definition of viscosity subsolution.
We say that $u$ is a viscosity solution if it is both a viscosity supersolution and subsolution.

Our main results in this section deal with 
minimizers of the energy functional $\energyfrac$.
In fact, we can also treat the larger class of one-sided minimizers, which are defined as follows:

\begin{definition}
	\label{def:one:sided}
	Given a bounded domain $\Omega \subset \R^n$, $s \in (0,1)$, and $F \in C^{1}(\R)$,
	we say that
	a function $u \colon \R^n \to \R$ 	
	is a \emph{one-sided minimizer by above} of the functional $\energyfrac$ in~$\Omega$ if 
	$\energyfrac(u)< \infty$ and, for all functions $v$ such that $v \geq u$ in $\Omega$ and $v = u$ in $\Omega^c$, 
	we have
	\[
	\energyfrac(v) \geq \energyfrac(u).
	\]
\end{definition}
We also have the analogous definition of one-sided minimizer by below.

We will now prove that one-sided minimizers by above are viscosity supersolutions.
This will follow from
Theorem \ref{thm:energy:comparison:frac}. As a consequence, minimizers are viscosity solutions.

\begin{theorem}
	\label{thm:viscosity:frac}
	Given a bounded domain $\Omega \subset \R^n$, $s \in (0,1)$, and $F \in C^{1}(\R)$, 
	let  
	$u \colon \R^n \to \R$ in $C(\Omega)$
	be a one-sided minimizer by above of the functional $\energyfrac$ in $\Omega$.
	
	Then, $u$ is a viscosity supersolution in $\Omega$.
\end{theorem}
\begin{proof}

	We argue by contradiction. 
	Suppose that $u$ is not a viscosity supersolution in~$\Omega$.
	Then there exist $x_0 \in \Omega$, a neighborhood $N \subset \Omega$ of $x_0$, and a function $\varphi \in C^2(\overline{N})$, with $\varphi(x_0) = u(x_0)$ and $\varphi(x) < u(x)$ 
	for all $x \in N \setminus \{x_0\}$,
	such that the extension $\overline{\varphi}$ defined by \eqref{extension-bar} satisfies
	\[
	(-\Delta)^s \overline{\varphi}(x_0) < F'(\overline{\varphi}(x_0)).
	\]
	
	We will now construct a function above $u$ which has less energy than $u$, thus violating the one-sided minimality by above. 
	The idea of the proof is to raise the function $\varphi$ to produce a local foliation whose leaves are strict subsolutions. The details go as follows.
	
	Recall the truncations introduced in the proof of Theorem~\ref{thm:energy:comparison:frac} above.
	Namely, for $\varepsilon > 0$ we let $K_{\varepsilon } = c_{n,s}|\cdot|^{-n-2s}  \mathds{1}_{\R^{n}\setminus B_{\varepsilon}}$ and $(-\Delta)^s_{\varepsilon} \varphi(x) = \int_{\R^n} (\varphi(x) - \varphi(y))K_{\varepsilon}(x-y) \d y$.
		
Since $(-\Delta)^s \overline{\varphi}(x_0) - F'(\overline{\varphi}(x_0)) =: -4 c_0 < 0$, by continuity of $F'$ and of the fractional Laplacian, 
there are some small $\delta > 0$ and $\varepsilon_0 > 0$ such that
$B_{2\delta}(x_0) \subset N$ and 
\[
(-\Delta)^s_{\varepsilon} \overline{\varphi}(x) - F'(\overline{\varphi}(x)) < -2 c_0,
\]
for all $x \in B_{2 \delta}(x_0)$ and $\varepsilon \in (0,\varepsilon_0)$.

	For $0\leq t \leq \tmax$, where $0 < \tmax \leq \min_{\overline{N} \setminus B_{\delta}(x_0)} (u-\varphi)$,
	we define the functions
	\[
	u^{t}(x) := 
	\left\{
	\begin{array}{ll}
		\max\{u(x),\varphi(x) + t\}& \text{ for } x \in N,\\
		u(x) & \text{ for } x \in \R^n \setminus N.\\
	\end{array}
	\right.
	\]
	Thanks to the choice $0 < t < \tmax \leq \min_{\overline{N} \setminus B_{\delta}(x_0)} (u-\varphi)$, 
	the functions $u^t$ are continuous (as required in condition \ref{def:visc:field:3} of weak fields).
	Notice that $\{u^t > u\} \subset B_{\delta}(x_0)$.

	It is clear that the family $\{u^{t}\}_{t \in [0,\tmax]}$ is a weak field for $u$ in $B_{\delta}(x_0)$ which is regular by below, in the sense of Definition~\ref{def:visc:field}.
	For this, notice for each $x \in \{u^t > u\}$, 
	the $C^2(\overline{N})$ function $\varphi + t $ touches $u^t$ by below in $B_{\delta}(x_0)$ at $x$.

	For $0 < t < \tmax$, $\varepsilon > 0$, and 
	$x \in \{u^t > u\}$,
	since $u^t(x) = \varphi(x) + t$
	we have
	\[
	\begin{split}
	(-\Delta)^s_{\varepsilon}\, u^{t}(x) &= \int_{\R^n} \left(\varphi(x) + t - u^t(y) \right) K_{\varepsilon}(x-y) \d y\\
	&= \int_{B_{\delta}(x)} \left(\varphi(x) + t - u^t(y) \right) K_{\varepsilon}(x-y) \d y \\
	& \hspace{3cm} + \int_{\R^n \setminus B_{\delta}(x)} \left(\varphi(x) + t - u^t(y) \right) K_{\varepsilon}(x-y) \d y\\
	&\leq \int_{B_{\delta}(x)} \left(\varphi(x)- \varphi(y) \right) K_{\varepsilon}(x-y) \d y \\
	& \hspace{3cm} + \int_{\R^n \setminus B_{\delta}(x)} \left(\varphi(x) - \overline{\varphi}(y) + t \right) K_{\varepsilon}(x-y) \d y,\\
	\end{split}	
	\]
	where in the last line we have used that $\varphi(y) + t \leq u^t(y)$ for  $ y \in B_{\delta}(x) \subset B_{2\delta}(x_0) \subset N$ 
	and $ \overline{\varphi}(y) \leq u^t(y)$ for $y \in \R^n$.
	Moreover, since $\overline{\varphi} = \varphi$ in $N$, it follows that 
	\[
	\begin{split}
	(-\Delta)^s_{\varepsilon}\, u^{t}(x) 
	&\leq \int_{\R^n} \left(\overline{\varphi}(x)- \overline{\varphi}(y) \right) K_{\varepsilon}(x-y) \d y  + t \int_{\R^n \setminus B_{\delta}(x)} K_{\varepsilon}(x-y) \d y\\
	&\leq (-\Delta)^s_{\varepsilon} \, \overline{\varphi}(x) +  \tmax c_{n,s} \int_{\R^n\setminus B_{\delta}(x)} |x-y|^{-n-2s} \d y\\
	& = (-\Delta)^s_{\varepsilon} \, \overline{\varphi}(x) +  \tmax c_{n,s} |\partial B_1| \frac{\delta^{-2s}}{2s}.
	\end{split}	
	\]
	From this inequality and the continuity of $F'$, taking a sufficiently small $\tmax$, we obtain
	\[
	(-\Delta)^s_{\varepsilon} u^{t}(x) - F'(u^{t}(x)) < -c_0,
	\]
	for all 
	$x \in B_{\delta}(x_0)$, 
	$\varepsilon \in (0,\varepsilon_0)$, and
	$t \in (0, \tmax)$ 
	such that $u^{t}(x) > u(x)$.
	
	Letting now $\varepsilon \downarrow 0$,
	by Theorem \ref{thm:energy:comparison:frac}, we conclude that
	\[
	\begin{split}
		\energyfrac(u^T)&\leq \energyfrac(u) + \int_{\Omega \cap \{u^T>u\}} \int_{u(x)}^{u^T(x)} \Big\{ (-\Delta)^s u^{t}(x) - F'(u^{t}(x))\Big\} \Big|_{t = t(x, \lambda)} \d \lambda \d x \\
		&\leq \energyfrac(u) -c_0 \big|\{(x,\lambda)\in \Omega\times \R : u(x)<\lambda < u^T(x)\}\big|\\
		&< \energyfrac(u),
	\end{split}
	\]
	which contradicts the minimality of $u$.
\end{proof}

\subsection{General nonlocal functionals} \label{subsection:viscosityGeneral}
\label{subsection:mixed}
We now extend the previous approach to the more general setting of mixed functionals
$\energytot$ of the form~\eqref{def:mixed}.
Here, 
properties~\ref{def:visc:field:1}-\ref{def:visc:field:4} from Definition~\ref{def:visc:field}
still yield the correct notion of weak field.
Under precise growth and regularity assumptions on the Lagrangians $\lag$ and $\lagloc$, 
the proofs below will be valid for sufficiently smooth weak fields.
The key point is to make sure that the mixed operator $\optot(u^{t})(x)$ is bounded by above, uniformly in $x$ and $t$.
In most applications, 
being $C^{1,1}$ by below
(property~\ref{def:visc:field:5} from Definition~\ref{def:visc:field} above)
suffices for the argument to go through.
Thus, in the next theorems,
by a weak field which is ``sufficiently regular for $\lagloc$ and $\lag$''
we mean a weak field satisfying the required regularity conditions.

Next, we prove the 
following energy comparison result in the presence of a sufficiently regular weak field.
It contains Theorem~\ref{thm:energy:comparison} in the Introduction
(and will be proven in a similar way to Theorem~\ref{thm:energy:comparison:frac} above).

\begin{theorem}
	\label{thm:energy:comparison_combined}
	Let $\lagloc = \lagloc(x,\lambda,q)$ be a function 
	satisfying 
	$\partial^2_{qq} \lagloc(x,\lambda,q)\geq 0$,
	and let~$\lag = \lag(x,y,a,b)$
	be a pairwise symmetric function 
	satisfying 
	$\partial^2_{ab} \lag(x,y,a,b) \leq 0$.
	
	Given a bounded domain $\Omega \subset \R^n$, let
	$u  \in C(\overline{\Omega})$.
	Assume that there exists $\{u^{t}\}_{t \in [0,\tmax]}$,
	a weak field for $u$ 
	in $\Omega$ (in the sense of Definition~\ref{def:visc:field}) which is sufficiently regular for~$\lagloc$ and $\lag$.
	
	Then, 
	if $\energytot(u) = \energyloc(u) + \energy(u) < +\infty$ (these functionals being defined in Section~\ref{section:compound}),
	we have
	\[
	\energytot(u^\tmax)\leq \energytot(u) + \int_{\Omega} \int_{u(x)}^{u^\tmax(x)} \optot\,(u^{t}(x)) \big|_{t = t(x, \lambda)} \d \lambda \d x,
	\]
	where $\optot = \oploc + \op$ is the Euler-Lagrange operator associated to $\energytot$.
\end{theorem}

\begin{proof}
	We consider the calibration functional constructed in Section~\ref{section:compound}, that is,
	\[
	\calibtot(w) = \energytot(u^\tmax) + \int_{\Omega} \int_{u^\tmax(x)}^{w(x)} \optot(u^{t})\big|_{t = t(x, \lambda)} \d \lambda \d x.
	\]
	Following the strategy there, one can show, by using the ellipticity conditions, that $\calibtot$ also satisfies property~\ref{def:calib:3} in the new framework of weak fields.\footnote{To prove this property for the nonlocal term, simply follow the analogous proof of Proposition~\ref{property:3} in Section~\ref{section:nonlocal}, replacing $\domain$ by $Q(\{u < u^\tmax\})$, similarly to the proof of Theorem~\ref{thm:energy:comparison:frac} above.}
	In particular,
	\[
	\energytot(u) \geq \calibtot(u)= \energytot(u^\tmax) + \int_{\Omega} \int_{u^\tmax(x)}^{u(x)} \optot(u^{t})\big|_{t = t(x, \lambda)} \d \lambda \d x,
	\]
	which yields the desired result.
\end{proof}

With this result at hand, we can easily show that one sided minimizers by above are viscosity supersolutions of the Euler-Lagrange equation.
Here it is clear how to adapt 
Definitions~\ref{def:viscosity} and~\ref{def:one:sided} 
to the case of mixed energy functionals.
The following result includes Theorem~\ref{thm:viscosity} in the Introduction (and is proven as Theorem~\ref{thm:viscosity:frac} above):
\begin{theorem}
	\label{thm:gen:visc:min}
	Let $\lagloc = \lagloc(x,\lambda,q)$ be a function satisfying 
	$\partial^2_{qq} \lagloc(x,\lambda,q)\geq 0$,
	and let
	$\lag = \lag(x,y,a,b)$
	be a pairwise symmetric function 
	satisfying 
	$\partial^2_{ab} \lag(x,y,a,b) \leq 0$.
	Let 
	$\Omega \subset \R^n$
	be a bounded domain and 
	let $u$ be a sufficiently regular one-sided minimizer by above of the functional $\energytot$.
	
	Then, the function u is a viscosity supersolution of the associated Euler-Lagrange equation $\optot(w) = 0$ in $\Omega$.
\end{theorem}
\begin{proof}
	Proceeding as in the proof of Theorem \ref{thm:viscosity:frac},
	we slide the touching function $\varphi$ upwards and take the maximum with $u$ to obtain a weak field.
	Applying Theorem \ref{thm:energy:comparison_combined}, we see that $\varphi$ cannot be a strict subsolution, since otherwise the leaves of the weak field would have smaller energy than the minimizer (this is shown as in the proof of Theorem~\ref{thm:viscosity:frac}).
\end{proof}

As a consequence of Theorem~\ref{thm:gen:visc:min}, we can finally give the proof of Theorem \ref{thm:viscosity} in the Introduction.

\begin{proof}[Proof of Theorem \ref{thm:viscosity}]
	If $u$ is a minimizer, then $u$ and $-u$ are one-sided minimizers by above of the functionals $\energy(\cdot)$ and $\energy(-\cdot)$, respectively.
	By Theorem \ref{thm:gen:visc:min}, the function $u$ is both a viscosity supersolution and subsolution.
	In particular, $u$ is a viscosity solution.
\end{proof}

\begin{remark}
	\label{remark:barron}
	For convex Lagrangians, there is a more direct proof of Theorem \ref{thm:gen:visc:min} which does not use calibrations.
	In addition to the ellipticity $\partial_{ab} \lag (x, y, a, b) \leq 0$,
	one needs to further assume that the functions $(\lambda,q) \mapsto \lagloc(x,\lambda,q)$ and $(a,b) \mapsto \lag(x, y, a, b)$ are both convex.
	The following argument is due to Barron and Jensen~\cite{BarronJensen}, who applied it to local functionals.
	Proceeding by contradiction as above,
	we consider the same weak field and pick one leaf $u^{t_0}$ with $t_{0} > 0$.
	By continuity and the ellipticity assumption, this function will satisfy $\optot (u^{t_{0}})(x) < 0$ whenever $u^{t_{0}}(x)>u(x)$.
	Then, applying integration by parts in the local term, symmetrizing in the nonlocal one, and using the convexity assumptions, we obtain
	\begin{align*}
		0&<\int_\Omega (u(x)-u^{t_{0}}(x)) \, \optot(u^{t_{0}})(x) \d x \\
		&\leq \int_{\Omega} \Big( \lagloc(x,u(x),\nabla u(x)) - \lagloc(x,u^{t_{0}}(x),\nabla u^{t_{0}}(x)) \Big)\d x \\
		&\ \ \ \ +\frac{1}{2} \iint_{Q(\Omega)} \Big( \lag(x,y,u(x),u(y)) - \lag(x,y,u^{t_0}(x),u^{t_0}(y)) \Big) \d x \d y \\
		&= 
		\energytot(u) - \energytot (u^{t_0}).
	\end{align*}
	This contradicts the one-sided minimality by below.
\end{remark}

\begin{remark}
	\label{remark:trick}
	The calibration approach allows us to prove that one-sided minimizers by above are viscosity supersolutions, 
	but a priori says nothing about those supersolutions which are not necessarily one-sided minimizers.
	Nevertheless, the strategy can be adapted to treat some nonlocal functionals whose main part is convex, for instance.
	The idea consists on building an auxiliary functional for which the weak supersolution is a one-sided minimizer, by interpreting lower order terms of the original functional as new linear terms.
	We briefly discuss the semilinear case for the sake of clarity.
	
	Let $u$ be a weak supersolution, not necessarily a one-sided minimizer, of the equation $(-\Delta)^s v = f(v)$ in $\Omega$. That is, $u$ satisfies
	$$ \frac{c_{n,s}}{2} \iint_{Q(\Omega)}  \frac{(u(x)-u(y))(\varphi(x)-\varphi(y)))}{|x-y|^{n+2s}} \d x \d y \geq \int_\Omega f(u(x))\,\varphi(x)\d x,$$
	for all $\varphi\in C^\infty_c(\Omega)$ such that $\varphi\geq 0$ in $\Omega$. 
	Then, one can check that $u$ is also a weak supersolution of the linear equation $(-\Delta)^s v = g$ in $\Omega$, with $g(x) := f(u(x))$. 
	In particular, $u$ is a one-sided minimizer by above of the auxiliary convex energy functional
	$$ \widetilde{\mathcal{E}}(w) = \frac{c_{n,s}}{4} \iint_{Q(\Omega)}  \frac{|w(x)-w(y)|^2}{|x-y|^{n+2s}}-\int_\Omega g(x) \, w(x). $$
	Applying Theorem \ref{thm:gen:visc:min}, we deduce that $u$ is a viscosity supersolution of the linear equation $(-\Delta)^s v = g$. 
	By definition of $g$, one clearly concludes that $u$ is a viscosity supersolution of the original semilinear equation.
\end{remark}

\appendix

\color{black}

\section{Minimality via the strong maximum principle}
\label{section:comparison}

In this appendix we explain how to prove minimality for a function embedded in a field of extremals via a strong comparison principle instead of via a calibration.
The proof, which is well known for local equations and for the fractional Laplacian, is simpler than building a calibration but requires (contrary to our calibration proof) an existence and regularity result for minimizers.
Such result will not be available for many nonlocal energy functionals.
This fact makes the calibration technique a stronger tool.

As in Section~\ref{section:nonlocal}, we let $\lag(x, y, a, b)$ be a nonlocal Lagrangian and the energy functional 
\[ \energy(w) = \dfrac{1}{2}\iint_{\domain} \lag(x, y, w(x), w(y)) \d x \d y. \]
As in the paper, we assume that $\lag$ is pairwise symmetric, that is, 
\[ \lag(y, x,b, a)= \lag(x, y, a, b) \quad \text{ for all } (x, y)\in \domain \text{ and } (a, b) \in \R^2.\]
The Euler-Lagrange operator associated to $\energy$ is given in terms of the integral
\[ 
\op(w)(x) = \int_{\R^{n}} \partial_{a} \lag(x, y, w(x), w(y)) \d y;
\]
see~\eqref{first:var}.

A sufficient condition for the operator $\op$ to satisfy a strong comparison principle is the strict ellipticity condition $\partial^2_{a b} \lag < 0$.
Indeed,
given two regular functions $u$ and $v$ defined in $\R^{n}$, if $u$ touches $v$ from below at some point $x_0$,  then the monotonicity of $\partial_{a} \lag$ leads to the inequality $ \op(u)(x_0) \geq \op(v)(x_0)$.
To see this, one must simply integrate $\partial_{a}\lag(x,y, u(x_0), u(y)) \geq \partial_{a}\lag(x,y, u(x_0), v(y))$ with respect to $y$ and use that $u(x_0) = v(x_0)$.  
Moreover, when 
$u \not\equiv v$
we have the strict inequality $\op(u)(x_0) > \op(v)(x_0)$.

We now prove the minimality result.
Let $\{u^t\}_{t \in I}$ be a field of extremals in $\Omega$ and suppose, arguing by contradiction, that $u^{t_0}$ with $t_0 \in I$ is not a minimizer.
Let $v$ be a minimizer of $\energy$ in the set of functions with ${\rm graph }\, v \subset \region = \{(x, \lambda) \in \R^{n}\times \R \,\colon\, \lambda = u^t(x) \text{ for some } t \in I\}$ satisfying the exterior condition $v = u^{t_0}$ in $\Omega^c$.
Assume further that $v$ is a continuous function.
In particular, $v \not\equiv u^{t_0}$ and by the monotonicity with respect to the leaf-parameter, there is a first leaf $u^{t_1}$ touching $v$ 
either from above or from below at an interior point $x_0 \in \Omega$ (since $v \equiv u^{t_0}$ outside $\Omega$).
In case it is from above (otherwise the argument is similar), the strong comparison principle now gives $0 = \op(v)(x_0) > \op(u^{t_1})(x_0) = 0$, which is a contradiction. Thus $u^{t_0} \equiv v$ is a minimizer. 

Note that this argument gives more than minimality. It establishes the uniqueness of the minimizer (and even of the extremal) with the given exterior condition.
It is also clear that the same argument works for fields made of super and subsolutions, that is, fields such that $\op(u^t) \geq 0$ for $t \geq t_0$ and $\op(u^t) \leq 0$ for $t \leq t_0$ in $\Omega$.

\section{The calibration for the nonlocal total variation}
\label{section:total:variation}

In this appendix
we relate our functional setting
to the geometric calibrations for the nonlocal perimeter appearing in the works of the first author~\cite{Cabre-Calibration} and of Pagliari~\cite{Pagliari}.
This is achieved through the \emph{nonlocal total variation} functional,
which amounts to
the integral of the nonlocal perimeters of the levels sets of a function.

Let us recall that, given an even kernel $K:\R^n\setminus \{0\} \to [0,+\infty)$, the $K$-nonlocal total variation of a function $w \colon \R^n \to \R$ is defined by
\begin{equation*}
\totalvar(w) := \frac{1}{2} \int \! \! \! \! \int_{Q(\Omega)} |w(x)-w(y)| \, K(x-y) \, \d x \d y,
\end{equation*} 
where $\Omega\subset \R^n$ is a bounded domain.
In particular, it is an energy functional of the form~\eqref{Intro_gen:energy} with Lagrangian
\[
\lag(x,y,a,b) = |a-b|\,K(x-y).
\]
Note that $\lag$ is elliptic and hence covered by our extremal field theory.\footnote{Notice that even though the Lagrangian $\lag$ is not $C^2$, it is convex in the $(a-b)$-variable, which suffices for our theory to apply.}

There is a strong connection between the nonlocal total variation and nonlocal minimal surfaces.
Here, we briefly explain some of the relevant ideas;
for more details in the particular case when $K(z) = |z|^{-n-s}$, we refer the reader to the work~\cite{BucurDipierroLombardiniValdinoci}.

Given a sufficiently regular set $E \subset \R^n$, its $K$-nonlocal perimeter is
\[
\perim(E) := \frac{1}{2} \int \! \! \! \! \int_{Q(\Omega)} |\mathds{1}_E(x) - \mathds{1}_E(y)| \, K(x-y) \, \d x \d y,
\]
and $E$ is called a $K$-nonlocal minimal surface if the first variation of $\perim$ at $E$ vanishes.
Notice that $\perim(E) = \totalvar( \mathds{1}_{E})$. 
It is well-known that the sublevel sets of minimizers of $\totalvar$ are $K$-nonlocal minimal surfaces.\footnote{Critical points of $\totalvar$ satisfy the equation $\mathcal{L}_{\rm NTV}(u)(x) := \int_{\R^n} {\rm sign}(u(x) - u(y)) K(x-y) \d y = 0$ for $x \in \Omega$. When $u(x)$ is a regular value, we have the relation $\mathcal{L}_{\rm NTV}(u)(x)= -H[\{u < u(x)\}](x)$, where $H[E] (x) = \int_{\R^n} \left( \mathds{1}_{E^c} (y)- \mathds{1}_{E} (y)\right) K(x-y) \d y$ denotes the nonlocal mean curvature of the set $E \subset \R^n$ at $x \in \partial E$. It follows that the sublevel sets are nonlocal minimal surfaces.}
Moreover, 
one can recover 
the nonlocal total variation $\totalvar$ 
of any function $w$ in terms of the nonlocal perimeter $\perim$ of its sublevel sets. 
Namely, we have the following nonlocal coarea formula~\cite{CintiSerraValdinoci}:
\begin{align} \label{coarea}
	\totalvar(w) = \int_{\R} \perim(\{w < \lambda\}) \d \lambda.
\end{align}

In~\cite{Pagliari},
Pagliari studied minimality properties of the nonlocal total variation $\totalvar$ when acting on functions taking values in the interval $[0,1]$.
He showed that characteristic functions of halfspaces minimize $\totalvar$ (among characteristic functions)
by constructing a calibration.
On the other hand, Cabr\'{e}~\cite{Cabre-Calibration} gave a calibration (recalled in~\eqref{perimcalib} below) for the $K$-nonlocal perimeter $\perim$ and an arbitrary set $E$
whenever it is embedded in a family of nonlocal minimal surfaces.\footnote{Notice that this covers the case of the half-space, simply by sliding it in the normal direction.}
Thus, the author extended the classical extremal field theory to nonlocal minimal surfaces.
Our present work is in the spirit of this second result but applied to the nonlocal total variation functional considered in~\cite{Pagliari}.
In particular, Theorem~\ref{thm:nonlocal:calibration} provides a calibration for the nonlocal total variation $\totalvar$ and an arbitrary function whenever it is embedded in a field of extremals.

Let us recall the calibration for the perimeter $\perim$ obtained in~\cite{Cabre-Calibration} by the first author.
Given a smooth function $\phi \colon \R^n \to \R$, for each $t \in \R$, we consider the superlevel sets $E^t = \{\phi(x) > t\}$.
In \cite{Cabre-Calibration} (see also \cite[Section~2]{CabreErnetaFelipeNavarro-Calibration1}), 
under the assumption that $E^t$ are nonlocal minimal surfaces,
it was shown that the functional
\begin{equation}
\label{perimcalib}
\calibper(F) = \frac{1}{2}\iint_{\domain} 
\text{\rm sign}
\left( \phi(x) - \phi(y) \right) \left(\mathds{1}_{F}(x) - \mathds{1}_{F}(y) \right) K(x-y) \d x \d y,
\end{equation}
is a calibration for $\perim$ and each $E^{t_0}$, $t_0 \in I$.

Next, we show that the analogue of the nonlocal coarea formula~\eqref{coarea} holds for the calibration functional.
Namely, the calibration for the nonlocal total variation functional constructed in the present paper can be written in terms of the calibration~\eqref{perimcalib} for the nonlocal perimeter of each sublevel set.
We point out that all identities in Proposition~\ref{prop:totalvar} hold for arbitrary fields $\{u^{t}\}_{t \in I}$, that is, the leaves $u^{t}$ are not necessarily extremals of $\totalvar$.

\begin{proposition}
\label{prop:totalvar}
Let $\{u^{t}\}_{t \in I}$ be a field in $\R^n$.
Then, the 
functional $\calibtotalvar$ associated to $\totalvar$ given by Theorem~\ref{thm:nonlocal:calibration}  can be written as
	\[
	\calibtotalvar(w) = \frac{1}{2} \int \! \! \! \! \int_{Q(\Omega)} \int_{w(y)}^{w(x)} \ \text{\rm sign}\left(u^{t(x,\lambda)}(x)-u^{t(x,\lambda)}(y)\right) \d \lambda \ K(x-y) \, \d x \d y.
	\]
	Moreover, the functional $\calibtotalvar$ can also be expressed as
\[
\calibtotalvar(w) = \int_\R 
\calibperla(\{w < \lambda\})
 \d \lambda,
\]
where $\calibperla$ is the calibration functional for the $K$-nonlocal perimeter $\perim$ in~\eqref{perimcalib}
constructed via 
the foliation given by the sublevel sets $ \{u^t < \lambda\}_{t \in I}$.
\end{proposition}

\begin{remark}
Before we succeeded in constructing a calibration for general functionals (even for the quadratic one in~\cite{CabreErnetaFelipeNavarro-Calibration1}),
we were able to build one
 for the nonlocal total variation~$\totalvar$.
For this, we considered the second identity in Proposition~\ref{prop:totalvar} as our definition of the calibration.
This idea was motivated by the coarea formula~\eqref{coarea}.
It is quite remarkable that our general construction in Theorem~\ref{thm:nonlocal:calibration} (found by completely different means)
recovers this natural calibration.
\end{remark}

\begin{remark}
We note that the sublevel sets $\{u^{t} < \lambda\}_{t \in I}$ appearing in the statement of Proposition~\ref{prop:totalvar} 
can be written as superlevel sets $E^{t}_{\lambda}:=\{\phi^{\lambda} > t\}_{t \in I}$ for some function~$\phi^{\lambda}$,
consistently with the notation for $\calibper$ in \eqref{perimcalib}.
Indeed, 
by monotonicity, we have
$u^t(x) < \lambda = u^{t(x,\lambda)}(x)$ if and only if $t < t(x, \lambda)$,
and hence we can take $\phi^{\lambda}(x) := t(x, \lambda)$.
That is, we have
\[
E^{t}_{\lambda} = \{x \in \R^n \colon t(x, \lambda) > t\} =\{u^{t} < \lambda\}.
\]
Moreover, 
again by the monotonicity of the field $\{u^t\}_{t\in I}$, 
for each $\lambda \in \R$, the level surfaces
$\partial E_{\lambda}^{t} = \partial \{u^{t} < \lambda\}$
give a foliation of a certain subset of $\R^n$.\footnote{We say that $\{\partial E^t\}_{t \in I}$ foliates a region $\mathcal{R}\subset \R^n$ if for each $x \in \mathcal{R}$, there is exactly one $t = t(x) \in I$ such that $x \in \partial E^t$. Following ideas from~\cite{Cabre-Calibration} and our previous work~\cite{CabreErnetaFelipeNavarro-Calibration1}, assuming $F = E^{t_0}$ outside $\Omega$, it is not hard to show that the calibration for the perimeter $\calibper$ may be written as \[ \calibper(F) = \perim(E^{t_0}) + \int_{\Omega} \left(\mathds{1}_{F \setminus E^{t_0}}(y) -\mathds{1}_{E^{t_0} \setminus F}(y) \right) H[E^{t}](x) \big|_{t = t(x)} \d x.\] In particular, when $\{\partial E^{t} \cap \Omega\}_{t \in I}$ foliates a subset $\mathcal{R} \subset \Omega$, the calibration properties hold for competitors $F$ with $F \triangle E^{t_0} := (F \setminus E^{t_0})\cup (E^{t_0} \setminus F) \subset \mathcal{R}$.} 
As previously discussed, since each $u^t$ is an extremal of the nonlocal total variation, the associated sublevel sets $E^{t}_{\lambda}$ are $K$-nonlocal minimal surfaces.
Therefore, the foliation and extremality properties of the family $\{u^{t}\}_{t \in I}$ are transferred to the sublevel sets $\{E^{t}_{\lambda}\}_{t \in I}$.
\end{remark}

\begin{proof}[Proof of Proposition~\ref{prop:totalvar}]
Let $t_0 \in I$.
First, letting $\lag(x,y,a,b) = |a-b|\,K(x-y)$, the calibration functional in Theorem~\ref{thm:nonlocal:calibration} associated to 
$\totalvar$
is
\begin{equation}
\label{miz0}
\calibtotalvar(w) = \int \! \! \! \! \int_{Q(\Omega)} \d x \d y \int_{u^{t_0}(x)}^{w(x)} \d \lambda \ \text{\rm sign}\left(u^{t(x,\lambda)}(x)-u^{t(x,\lambda)}(y)\right) K(x-y) + \totalvar(u^{t_0}).
\end{equation}

It is easy to check that the sign factor in \eqref{miz0} can be written as
\[
\text{\rm sign}\left(u^{t(x,\lambda)}(x)-u^{t(x,\lambda)}(y)\right) = \text{\rm sign}\left(t(y,\lambda)-t(x,\lambda)\right),
\]
for all $x$ and $y$ in $\R^n$ and $\lambda \in \R$.
Symmetrizing 
the first term in the right-hand side of~\eqref{miz0}
in the variables $x$ and $y$, 
and then using that 
\[
\int_{u^{t_0}(x)}^{w(x)} \cdot \d \lambda - \int_{u^{t_0}(y)}^{w(y)} \cdot \d \lambda = \int_{w(y)}^{w(x)} \cdot \d \lambda - \int_{u^{t_0}(y)}^{u^{t_0}(x)}  \cdot \d \lambda,
\]
we see that
\begin{equation}
\label{miz1}
\begin{split}
& \int \! \! \! \! \int_{Q(\Omega)} \d x \d y \int_{u^{t_0}(x)}^{w(x)} \d \lambda \ \text{\rm sign}\left( u^{t(x,\lambda)}(x)-u^{t(x,\lambda)}(y) \right) \ K(x-y) \\
		& \quad  \quad  \quad = \frac{1}{2} \int \! \! \! \! \int_{Q(\Omega)} \d x \d y \int_{w(y)}^{w(x)} \d \lambda \ \text{\rm sign}\left(t(y,\lambda)-t(x,\lambda)\right) \ K(x-y) \\
		& \quad \quad  \quad \quad\quad \quad  -\frac{1}{2} \int \! \! \! \! \int_{Q(\Omega)} \d x \d y \int_{u^{t_0}(y)}^{u^{t_0}(x)} \d \lambda \ \text{\rm sign}\left(t(y,\lambda)-t(x,\lambda)\right) \ K(x-y).
\end{split}
\end{equation}
On the other hand, by the nonlocal coarea formula \eqref{coarea} and the simple identity
\[
|\mathds{1}_{\{u^{t_0}<\lambda\}}(x) - \mathds{1}_{\{u^{t_0}<\lambda\}}(y)| = 
\text{\rm sign}\left(t(x,\lambda)-t(y,\lambda)\right)\left(\mathds{1}_{\{u^{t_0} < \lambda\}}(x) - \mathds{1}_{\{u^{t_0} <\lambda\}}(y)\right),
\]
it is not hard to show that
\begin{equation}
\label{miz2}
\begin{split}
\totalvar(u^{t_0})  &= \frac{1}{2} \int \! \! \! \! \int_{Q(\Omega)} \d x \d y \int_{\R} \d \lambda \ |\mathds{1}_{\{u^{t_0}<\lambda\}}(x) - \mathds{1}_{\{u^{t_0}<\lambda\}}(y)|  K(x-y) \\
&=\frac{1}{2} \int \! \! \! \! \int_{Q(\Omega)} \d x \d y \int_{u^{t_0}(y)}^{u^{t_0}(x)} \d \lambda \ \text{\rm sign}\left(t(y,\lambda)-t(x,\lambda)\right) K(x-y).
\end{split}
\end{equation}
Combining \eqref{miz1} and \eqref{miz2}, 
from \eqref{miz0} we deduce
\[
\begin{split}
		\calibtotalvar(w) &= \frac{1}{2} \int \! \! \! \! \int_{Q(\Omega)} \d x \d y \int_{w(y)}^{w(x)} \d \lambda \ \text{\rm sign}\left(t(y,\lambda)-t(x,\lambda)\right) \ K(x-y)\\
		&=\frac{1}{2} \int \! \! \! \! \int_{Q(\Omega)} \d x \d y \int_{w(y)}^{w(x)} \d \lambda \ \text{\rm sign}\left(u^{t(x,\lambda)}(x)-u^{t(x,\lambda)}(y)\right) \ K(x-y),
\end{split}
\]
which is the first claim of the proposition.

Finally, the last expression for $\calibtotalvar$
can also be written as
\[
\begin{split}
		\calibtotalvar(w) &= 
		\frac{1}{2} 
		\iint_{\domain}\! \! \! \! \! \! \d x \d y \int_{\R} \d \lambda \ \text{\rm sign}\left(t(x,\lambda)-t(y,\lambda)\right) (\mathds{1}_{\{w< \lambda\}}(x) - \mathds{1}_{\{w< \lambda\}}(y)) K(x-y).
\end{split}
\]
Changing the order of integration, to finish the proof it remains to show that
\[
\calibperla(\{w < \lambda\}) = \frac{1}{2} \iint_{\domain} \!\!\! \text{\rm sign}\left(t(x,\lambda)-t(y,\lambda)\right) (\mathds{1}_{\{w< \lambda\}}(x) - \mathds{1}_{\{w< \lambda\}}(y)) K(x-y) \d x \d y.
\]
But this is precisely the calibration $\calibper$ in \eqref{perimcalib} with $\phi(x) = \phi^{\lambda}(x) =  t(x, \lambda)$. 
This yields the claim.
\end{proof}

\section*{Acknowledgments}

The authors thank Joaquim Serra for suggesting the topic of viscosity solutions.
They also wish to thank the anonymous referee for their careful reading of the manuscript, for their useful suggestions, and for bringing to our attention reference~\cite{BucurDipierroLombardiniValdinoci}.

\end{document}